\documentclass[10pt]{article}
\usepackage{amsmath, amssymb ,amsthm, amsfonts, amsgen, color}
\usepackage{graphicx}
\usepackage{bbm}
\usepackage{tikz}
\usepackage[colorinlistoftodos]{todonotes}
\usepackage{hyperref}
\usepackage[bottom]{footmisc}

\numberwithin{equation}{section} 
\definecolor{vg}{rgb}{0.0, 0.26, 0.15}

\newcommand{\Nb}{{\mathbb{N}}}
\newcommand{\Rb}{{\mathbb{R}}}

\newcommand{\wsto}{\stackrel{*}{\rightharpoonup}}

\newcommand{\E}{{\mathcal{E}}}

\newcommand{\cH}{{\mathcal{H}}}

\newcommand{\LL}{{\mathcal{L}}}

\makeatletter
\def\rightharpoonupfill@{\arrowfill@\relbar\relbar\rightharpoonup}
\newcommand{\xrightharpoonup}[2][]{\ext@arrow
0359\rightharpoonupfill@{#1}{#2}} \makeatother

\def\e{{\varepsilon}}
\def\O{{\Omega}}

\def\E{{\cal E}}

\def\L{{\cal L}}

\def\E{{\exists}}

\def\E{{\cal E}}

\newtheorem{theorem}{Theorem}[section]
\newtheorem{lemma}[theorem]{Lemma}
\newtheorem{proposition}[theorem]{Proposition}

\newtheorem{remark}[theorem]{Remark}
\newtheorem{definition}[theorem]{Definition}

\newcommand{{\rr}}{{\mathbb R}}

\newenvironment{@abssec}[1]{%
     \if@twocolumn
       \section*{#1}%
     \else
       \vspace{.05in}\footnotesize
       \parindent .2in
         {\upshape\bfseries #1. }\ignorespaces
     \fi}
     {\if@twocolumn\else\par\vspace{.1in}\fi}

\begin{document}

\title{\sc Relaxation for an optimal design problem in $BD(\O)$}

\author{{\sc Ana Cristina Barroso}\\
Departamento de Matem\'atica and CMAFcIO \\ 
Faculdade de Ci\^encias da Universidade de Lisboa\\ 
Campo Grande, Edif\' \i cio C6, Piso 1\\
1749-016 Lisboa, Portugal\\
acbarroso@ciencias.ulisboa.pt\\
 \\
{\sc Jos\'e Matias}\\
Departamento de Matem\'{a}tica and CAMGSD\\
Instituto Suprior T\'{e}cnico\\
Av. Rovisco Pais 1\\
1049-001 Lisboa, Portugal\\
jose.c.matias@tecnico.ulisboa.pt\\
\\
and
\\
\\
{\sc Elvira Zappale} \\
Dipartimento di Scienze di Base e Applicate per l'Ingegneria\\
Sapienza - Universit\`a di Roma\\
Via Antonio Scarpa, 16\\
00161 Roma (RM), Italy\\
elvira.zappale@uniroma1.it}
\maketitle

\begin{abstract}
We obtain a measure representation for a functional arising in the context of optimal design problems under linear growth conditions. The functional in question corresponds to the relaxation with respect to a pair $(\chi,u)$, where $\chi$ is the characteristic function of a set of finite perimeter and $u$ is a function of bounded deformation, of an energy with a bulk term depending on the symmetrised gradient as well as a perimeter term. 
\end{abstract}

\smallskip



\section{Introduction}

In optimal design one aims to find an optimal shape which minimises a cost functional. The optimal shape is a subset $E$ of a bounded, open set 
$\Omega \subset \mathbb R^N$ which is described by its characteristic function 
$\chi: \Omega \to \{0,1\}$, $E = \{\chi= 1\}$, and, in the linear elasticity framework, the cost functional is usually a quadratic energy, so we are lead to the problem
\begin{equation}\label{orpb} 
\displaystyle \min_{(\chi,u)} \int_{\Omega}\chi(x) W_1({\cal E} u(x)) 
	+ (1 - \chi(x))W_0(\E u(x)) \, dx,
\end{equation}
where $W_0$ and $W_1$ are two {\em elastic} densities, with $W_0 \geq W_1$, and $\E u$ denotes the symmetrised gradient of the displacement $u$. We refer to the seminal papers \cite{AL, KS1, KS2, KS3, MT}, among a wide literature (see, for instance, the recent contributions \cite{BIR} and \cite{BIR2}).

However, as soon as plasticity comes into play, the
observed stress-strain relation is no longer linear and, due to the linear growth of the stored elastic energy and to the lack of reflexivity of the space $L^1$, a suitable functional space is necessary to account for fields $u$ whose strains are measures. The space of special fields with bounded deformation, $BD(\Omega)$, was first proposed in \cite{S1}-\cite{S4}
and starting from these pioneering papers a vast literature developed. 

Indeed, already in the case where $\chi\equiv \chi_\Omega$,
the search for equilibria in the context of perfect plasticity leads naturally to the study of lower semicontinuity properties, and eventually relaxation, for energies of the type
\begin{equation}\label{calE}
	\int_{\Omega}f(\E u(x)) \, dx
\end{equation}
where 
$f$ is the volume energy density. As mentioned above, $u$ belongs to the space $BD(\Omega)$ of functions of bounded deformation composed of integrable vector-valued functions for which all components $E_{ij}$, $i,j = 1, \ldots,N$, of the deformation tensor $Eu := \frac{Du + Du^T}{2}$ are bounded Radon measures and 
$\E u$ stands for the absolutely continuous part, with respect to the Lebesgue measure, of the symmetrised distributional derivative $Eu$. 

Lower semicontinuity for \eqref{calE} was established in \cite{BCDM} under convexity assumptions on $f$ and in \cite{E} for symmetric quasiconvex integrands, under linear growth conditions and for $u \in LD(\Omega)$, the subspace of $BD(\Omega)$ comprised of functions for which the singular part $E^su$ of the measure $Eu$ vanishes. For a symmetric quasiconvex density $f$ with an explicit dependence on the position in the body and satisfying superlinear growth assumptions, lower semicontinuity properties were established in \cite{E2} for $u \in SBD(\Omega)$.

In the case where the energy density takes the form $\|\E u\|^2$ or 
$\|\E^D u\|^2 + (\text{div } u)^2$ (where $A^D$ stands for the deviator of the 
$N \times N$ matrix $A$ given by $A^D := A - \frac{1}{N}\text{tr} (A)I$), and the total energy also includes a surface term, a first relaxation result was proved in \cite{BDV}. We also refer to  \cite{KR} for the relaxation in the case where there is no surface energy and to \cite{Mo}, \cite{JS},  and \cite{MMS} for related models concerning evolutions and homogenization, among a wider list of contributions.

For general energy densities $f$, Barroso, Fonseca $\&$ Toader \cite{BFT} studied the relaxation of \eqref{calE} for $u \in SBD(\Omega)$ under linear growth conditions but placing no convexity assumptions on $f$. They showed that the relaxed functional admits an integral representation where a surface energy term arises naturally. The global method for relaxation due to Bouchitté, Fonseca $\&$ Mascarenhas \cite{BFM} was used to characterise the density of this term, whereas the identification of the relaxed bulk energy term relied on the blow-up method  \cite{FM1} together with a Poincaré-type inequality.

Ebobisse $\&$ Toader \cite{ET} obtained an integral representation result for general local functionals defined in $SBD(\Omega)$ which are lower semicontinous with respect to the $L^1$ topology and satisfy linear growth and coercivity conditions. The functionals under consideration are restrictions of Radon measures and are assumed to be invariant with respect to rigid motions. Their work was extended to the space $SBD^p(\Omega)$, $p > 1$, which arises in connection with the study of fracture and damage models, by Conti, Focardi $\&$ Iurlano \cite{CFI} in the 2-dimensional setting. A crucial and novel ingredient of their proof is the construction of a $W^{1,p}$ approximation of an $SBD^p$ function $u$ using finite-elements on a countable mesh which is chosen according to $u$ (recall that $SBD^p$ denotes the space of fields with bounded deformation such that the symmetrised gradient is the sum of an $L^p$ field and a measure supported on a set of finite $\mathcal H^{N-1}$ measure).

The analysis of an integral representation for a variational functional satisfying lower semicontinuity, linear growth conditions and the usual measure theoretical properties, was extended to the full space $BD(\Omega)$ by Caroccia, Focardi $\&$ Van Goethem \cite{CFVG}. In this work, the invariance of the studied functional with respect to rigid motions, required in \cite{ET}, is replaced by a weaker condition stating continuity with respect to infinitesimal rigid motions. Their result relies, as in papers mentioned above, on the global method for relaxation, as well as on the characterisation of the Cantor part of the measure $Eu$, due to De Philippis $\&$ Rindler \cite{DPR0}, which extends to the $BD$ case the result of Alberti's rank-one theorem in $BV$. 

In the study of the minimisation problem \eqref{orpb} one usually prescribes the volume fraction of the optimal shape, leading to a constraint of the form
$\displaystyle \frac{1}{{\mathcal L}^N(\Omega)}\int_\Omega \chi(x) dx= \theta,  \;\theta \in (0,1)$. It is sometimes convenient to replace this constraint 
by inserting, instead, a Lagrange multiplier in the modelling functional
which, in the optimal design context, becomes
 \begin{equation}\label{functnoper}F(\chi,u;\Omega) : = 
 \displaystyle \int_{\Omega}\chi(x) W_1({\cal E} u(x)) 
 + (1 - \chi(x))W_0(\E u(x)) + \int_\Omega k \chi(x) dx. 
 \end{equation}

Despite the fact that we have compactness for $u$ in $BD(\Omega)$ for functionals of the form \eqref{functnoper}, it is well known that the problem
of minimising \eqref{functnoper} with respect to $(\chi,u)$, adding suitable forces and/or boundary conditions, is ill-posed, in the sense that
minimising sequences $\chi_n \in L^\infty(\Omega;\{0,1\})$ tend to highly oscillate and develop microstructure, so that in the limit we may no longer obtain a characteristic function. To avoid this phenomenon, as in \cite{ABper} and \cite{KL}, we add a perimeter penalisation along the interface between the two zones 
$\{\chi = 0\}$ and $\{\chi = 1\}$ (see \cite{CZ} for the analogous analysis performed in $BV$, and \cite{CZ0, BZ, BZ2} for the Sobolev settings, also in the presence of a gap in the growth exponents). 
 
Thus, with an abuse of notation (i.e. denoting $W_1+ k$, in \eqref{functnoper}, still by $W_1$), our aim in this paper is to study the energy functional given by
 \begin{equation}\label{Fint}
 	F(\chi,u;\Omega) : = 
 	\displaystyle \int_{\Omega}\chi(x) W_1({\cal E} u(x)) 
 	+ (1 - \chi(x))W_0(\E u(x)) \, dx + |D\chi|(\Omega), 
 \end{equation}
 where $u \in BD(\Omega)$, $\chi \in BV(\Omega;\{0,1\})$ and the densities $W_i$, $i = 0,1$, are continuous functions satisfying the following linear growth conditions from above and below,
 \begin{equation}\label{growthint}
 	\exists \, \alpha, \beta > 0 \text{ such that }
 	\alpha |\xi| \leq W_i(\xi) \leq \beta (1 + |\xi|), \; \; \forall 
 	\xi \in \Rb^{N \times N}_s.
 \end{equation}
 We point out that no convexity assumptions are placed on $W_i$, $i = 0,1$.
 
To simplify the notation, in the sequel, we let 
$f: \{0,1\} \times \mathbb R^{N \times N}_s \to [0,+\infty)$
be defined as
\begin{equation}\label{densityint}
f\left( q,\xi\right)  :=q  W_1(\xi)+ (1-q)W_0(\xi), 
\end{equation}
and for a fixed $q \in\{0,1\}$, we recall that the recession function of $f$, in its second argument, is given by
\begin{equation}\label{recS}
f^{\infty}(q,\xi) := \limsup_{t\rightarrow+\infty}\frac{f(q,t\xi)}{t}.
\end{equation}

Since we place no convexity assumptions on $W_i$, we consider
the relaxed localised functionals arising from the energy \eqref{Fint}, 
defined, for an open subset $A \subset \Omega$, by 
\begin{align}\label{calFint}
\mathcal{F}\left(\chi,u;A\right)   &  
:=\inf\bigg\{ \liminf_{n\rightarrow +\infty} F(\chi_n,u_n;A): 
u_{n} \in W^{1,1}(A;\mathbb{R}^{N}),
\chi_{n} \in BV(A;\{0,1\}), \\
&  \hspace{4cm} 
u_{n}\to u\text{ in }
L^{1}(A;\mathbb{R}^{N}),
\chi_{n}\to\chi\text{ in }
L^1(A;\{0,1\})  \bigg\},\nonumber
\end{align}
and
\begin{align}\label{calFLDint}
\mathcal{F}_{LD}\left(\chi,u;A\right)   &  
:=\inf\bigg\{ \liminf_{n\rightarrow +\infty} F(\chi_n,u_n;A): 
u_{n} \in LD(A), \chi_{n} \in BV(A;\{0,1\}), \\
&  \hspace{4cm} 
u_{n}\to u\text{ in }
L^{1}(A;\mathbb{R}^{N}),
\chi_{n}\to\chi\text{ in }
L^1(A;\{0,1\})  \bigg\},\nonumber
\end{align}
where $LD(\Omega):=\{u \in BD(\Omega): E^s u=0\}$.

Due to the expression of \eqref{Fint}, and to the fact that 
$\chi_n \overset{\ast}{\rightharpoonup}\chi$ in $BV$ if and only if $\{\chi_n\}$ is uniformly bounded in $BV$ and $\chi_n \to \chi$ in $L^1$, it is equivalent to take
$\chi_n \overset{\ast}{\rightharpoonup}\chi$ in $BV$ or $\chi_n \to \chi$ in $L^1$
in the definitions of the functionals \eqref{calFint} and \eqref{calFLDint}, obtaining for each of them the same infimum regardless of the considered convergence.

As a simple consequence of the density of smooth functions in $LD(\Omega)$ we show in Remark~\ref{remkW11LD} that, under the above growth conditions on $W_0, W_1$, 
$$\mathcal{F}\left(\chi,u;A\right) = \mathcal{F}_{LD}\left(\chi,u;A\right), \;
\mbox{for every} \; \chi \in BV(A;\left\{0,1\right\}), u \in BD(\Omega),
A \in {\cal O}(\Omega).$$
We prove in Proposition~\ref{traceRm} that $\mathcal{F}\left(\chi,u;\cdot\right)$ is the restriction to the open subsets of $\Omega$ of a Radon measure, the main result of our paper concerns the characterisation of this measure.

\begin{theorem}\label{main}
Let $f:\{0,1\} \times \mathbb R^{N \times N}_s\to [0, + \infty)$ be a continuous function as in \eqref{densityint}, where $W_0$ and $W_1$ satisfy \eqref{growthint}, and consider 
$F:BV(\Omega;\{0,1\})\times BD(\Omega)\times \mathcal O(\Omega)$ 
defined in \eqref{Fint}. Then
\begin{align}\label{F11}
\mathcal{F}\left(  \chi,u;A\right)&=\int_A SQf(\chi(x), \E u(x)) \, dx + 
\int_{A \cap (J_\chi \cup J_u)} 
g(x,\chi^+(x),\chi^-(x),u^+(x),u^-(x),\nu(x))\, d\cH^{N-1}(x) \nonumber \\
& + \int_A(SQf)^{\infty}(\chi(x), \frac{d E^c u}{d |E^c u|}(x)) \, d|E^c u|(x),
\end{align}
where $SQf$ is the symmetric quasiconvex envelope of $f$ and 
$(SQf)^{\infty}$ is its recession function (cf. Subsection~\ref{defqcx} and \eqref{recS}, respectively). The relaxed surface energy density is given by
$$
g(x_0,a,b,c,d,\nu) := \limsup_{\varepsilon \to 0^+}
\frac{m(\chi_{a,b,\nu}(\cdot - x_0),u_{c,d,\nu}(\cdot - x_0);Q_\nu(x_0,\e))}
{\e^{N-1}}
$$
where $Q_{\nu}(x_0,\e)$ stands for an open cube with centre $x_0$, sidelength $\e$ and two of its faces parallel to the unit vector $\nu$, 
$$m(\chi,u;V) := \inf\left\{\mathcal F (\theta,v;V) : 
\theta \in  BV(\O;\{0,1\}), v \in BD(\O), \theta = \chi \mbox{ on } \partial V,
v = u \mbox{ on } \partial V\right\},$$
for any $V$ open subset of $\Omega$ with Lipschitz boundary,
and, for $(a,b,c,d,\nu)  \in \{0,1\} \times \{0,1\}  \times \mathbb{R}^{N} \times\mathbb{R}^{N} \times S^{N-1},$ the functions $\chi_{a,b,\nu}$ and
$u_{c,d,\nu}$ are defined as
\begin{equation*}
\chi_{a,b,\nu}(y) := \begin{cases}
a, & {\rm if } \; y \cdot \nu > 0\\
b, & {\rm if } \; y \cdot \nu < 0
\end{cases}
\; \; \; {\rm and } \; \; \;
u_{c,d,\nu}(y) := \begin{cases}
c, & {\rm if } \; y \cdot \nu > 0\\
d, & {\rm if } \; y \cdot \nu < 0.
\end{cases}
\end{equation*}
 \end{theorem} 

For the notation regarding the jump sets $J_\chi$, $J_u$ and the corresponding vectors
$\chi^+(x)$, $\chi^-(x)$, $\nu_\chi(x)$, $u^+(x)$, $u^-(x)$ and $\nu_u(x)$ we refer to Subsections~\ref{BV}, \ref{BD} and \ref{surface}.

The above expression for the relaxed surface energy density arises as an application of the global method for relaxation \cite{BFM}. However, as we will see in Subsection~\ref{surface}, in the case where $f$ satisfies the additional hypothesis \eqref{finfty}, this density can be described more explicitly, leading to an integral representation for \eqref{calFint}, in the $BD$ setting, entirely similar to the one in $BV$, obtained in \cite{CZ}, when $W_0$ and $W_1$ depend on the whole gradient $\nabla u$. Indeed, under this assumption, we show that
$$g(x_0,a,b,c,d,\nu) = K(a,b,c,d,\nu)$$
where
\begin{equation}\label{KSQ}
K(a,b,c,d,\nu):=\inf\left\{  
	\displaystyle\int_{Q_{\nu}}
	(SQf)^{\infty}(\chi(x),\E u(x)) \, dx+|D\chi|(Q_{\nu}):\left(\chi,u\right)
\in\mathcal{A}(a,b,c,d,\nu)\right\}, 
\end{equation}
and, for $(a,b,c,d,\nu)  \in \{0,1\} \times \{0,1\}  \times \mathbb{R}^{N} \times\mathbb{R}^{N} \times S^{N-1},$ the set of admissible functions is
\begin{align*}
\mathcal{A}(a,b,c,d,\nu)   &  :=\bigg\{  \left(\chi,u\right)
\in BV_{\rm loc}\left(S_{\nu};\{0,1\}\right)  \times 
W^{1,1}_{\rm loc}\left(S_{\nu};\mathbb{R}^{N}\right)  : \nonumber \\
&   (\chi(y),u(y)) = (a,c) \text{ if } y\cdot\nu=\frac{1}{2},
~(\chi(y),u(y)) = (b,d) \text{ if } y\cdot\nu=-\frac{1}{2},\nonumber\\
&   (\chi, u)\text{ are 1-periodic in the directions of }\nu_{1},\dots,\nu_{N-1}\bigg\},
\end{align*}
$\left\{\nu_{1},\nu_{2},\dots,\nu_{N-1},\nu\right\}$ is an orthonormal
basis of $\mathbb{R}^{N}$ and $S_\nu$ is the strip given by
$$S_\nu = \left\{x \in \Rb^N : |x \cdot \nu| < \frac{1}{2}\right\}.$$

As an application of the result of Caroccia, Focardi $\&$ Van Goethem, obtained in the abstract variational functional setting in \cite{CFVG}, the authors proved an integral representation for the relaxed functional, defined in $BD(\Omega) \times \mathcal O(\Omega)$,
$$
\mathcal F_{0}(u; A) := \inf \Big \{\liminf_{n \to + \infty} 
F_0(u_n;A) :  u_n \in W^{1, 1}(A; \Rb^N), u_n \to u \; {\rm in} \; L^1(A;\Rb^N)
\Big \},$$
where
$$F_0(u;A) :=\begin{cases}
\displaystyle \int_A f_{0} (x,u(x),\E u_n(x)) \, dx, & {\rm if} \; u \in W^{1,1}(\Omega;\Rb^N)\\
+\infty, & {\rm otherwise}
\end{cases}$$
and the density $f_0$ satisfies linear growth conditions from above and below
$$\frac{1}{C}|A| \leq f_0(x,u,A) \leq C\left(1 + |A|\right), \; \forall (x,u,A) \in
\Omega \times \Rb^N \times \Rb^{N \times N}_s$$
as well as a continuity condition with respect to $(x,u)$. This generalises to the full space $BD(\Omega)$, and to the case of densities $f_0$ depending explicitly on $(x,u)$, the results obtained in \cite{BFT}. We will make use of their work in Subsection~\ref{Cantor} to prove both lower and upper bounds for the density of the Cantor part of the measure 
$\mathcal F(\chi,u;\cdot)$, by means of an argument based on Chacon's Biting Lemma which allows us to fix $\chi$ at an appropriately chosen point $x_0$, as in \cite{MMZ}.

The contents of this paper are organised as follows. In Section~\ref{prelim}
we fix our notation and provide some results pertaining to $BV$ and $BD$ functions and notions of quasiconvexity 
which will be used in the sequel.
Section~\ref{auxres} contains some auxiliary results which are needed to prove our main theorem. In particular, in Proposition~\ref{traceRm} we show that $\mathcal F(\chi,u;\cdot)$ is the restriction to the open subsets of $\Omega$ of a Radon measure $\mu$. Section~\ref{mainthm} is dedicated to the proof of our main theorem, which characterises this measure. In each of Subsections~\ref{bulk}, \ref{Cantor} and \ref{surface} we prove lower and upper bounds of the densities of $\mu$ with respect to the bulk and Cantor parts of $Eu$, as well as with respect to a surface measure which is concentrated on the union of the jump sets of $\chi$ and $u$.

The fact that our functionals have an explicit dependence on the $\chi$ field
prevented us from applying existing results (such as \cite{ARDPR} and \cite{BDG})  directly and required us to obtain direct proofs.

\section{Preliminaries}\label{prelim}

In this section we fix notations and quote some definitions and 
results that will be used in the sequel.

Throughout the text $\Omega \subset \mathbb R^{N}$ will denote an open, bounded set
with Lipschitz boundary.

We will use the following notations:
\begin{itemize}
\item ${\mathcal B}(\Omega)$, ${\mathcal O}(\Omega)$ and ${\mathcal O}_{\infty}(\Omega)$ represent the families of all Borel, open and open subsets of $\Omega$ with Lipschitz boundary, respectively;
\item $\mathcal M (\Omega)$ is the set of finite Radon
measures on $\Omega$;
\item $\left |\mu \right |$ stands for the total variation of a measure  $\mu\in \mathcal M (\Omega)$; 
\item $\mathcal L^{N}$ and $\mathcal H^{N-1}$ stand for the  $N$-dimensional Lebesgue measure 
and the $\left(  N-1\right)$-dimensional Hausdorff measure in $\mathbb R^N$, respectively;
\item the symbol $d x$ will also be used to denote integration with respect to $\mathcal L^{N}$;
\item the set of symmetric $N \times N$ matrices is denoted by $\Rb_s^{N\times N}$;
\item given two vectors $a, b \in \Rb^N$, $a \odot b$ is the symmetric $N \times N$ matrix defined by $a \odot b : = \dfrac{a \otimes b + b \otimes a}{2}$, where $\otimes$ indicates tensor product;
\item $B(x, \e)$ is the open ball in $\Rb^N$ with centre $x$ and radius $\e$, 
$Q(x,\e)$ is the open cube in $\Rb^N$ with two of its faces parallel to the unit vector $e_N$, centre $x$ and sidelength $\e$, whereas
$Q_{\nu}(x,\e)$ stands for a cube with two of its faces parallel to the unit vector $\nu$; when $x=0$ and $\e = 1$, $\nu=e_N$ we simply write $B$ and $Q$;
\item $S^{N-1} := \partial B$ is the unit sphere in $\Rb^N$;
\item $C_c^{\infty}(\Omega; \Rb^N)$ and $C_{\rm per}^{\infty}(\Omega; \Rb^N)$ are the spaces of $\Rb^N$-valued smooth functions with compact support in $\Omega$ and
smooth and $Q$-periodic functions from $Q$ to $\Rb^N$, respectively;
\item by 
$\displaystyle \lim_{\delta, n}$ we mean 
$\displaystyle \lim_{\delta \to 0^+} \lim_{n\to +\infty}$,
$\displaystyle \lim_{k, n}$ means 
$\displaystyle \lim_{k \to +\infty} \lim_{n\to +\infty}$;
\item  $C$ represents a generic positive constant that may change from line to line.
\end{itemize}

\subsection{BV Functions and Sets of Finite Perimeter}\label{BV}

In the following we give some preliminary notions regarding functions of bounded variation and sets of finite perimeter. For a detailed treatment we refer to \cite{AFP}.

\smallskip

Given $u \in L^1(\Omega; \Rb^d)$ we let $\O _u$ be the set of Lebesgue points of $u$,
i.e., $x\in \O _u$ if there exists $\widetilde u(x)\in {\mathbb{R}}^d$ such
that 
\begin{equation*}
\lim_{\varepsilon\to 0^+} \frac{1}{\e^N}\int_{B(x,\varepsilon)}
|u(y)-\widetilde u(x)|\,dy=0,
\end{equation*}
$\widetilde u(x)$ is called the approximate limit of $u$ at $x$.
The Lebesgue discontinuity set $S_u$ of $u$ is defined as 
$S_u := \Omega \setminus \O _u$. It is known that ${\mathcal{L}}^{N}(S_u) = 0$ and the function 
$x \in \Omega \mapsto \widetilde u(x)$, which coincides with $u$ $\L ^N$- a.e.
in $\O _u$, is called the Lebesgue representative of $u$.

The jump set of the function $u$, denoted by $J_u$, is the set of
points $x\in \O \setminus \O _u$ for which there exist 
$a, \,b\in {\mathbb{R}}^d$ and a unit vector $\nu \in S^{N-1}$, normal to $J_u$ at $x$, such that $a\neq b$ and 
\begin{equation*}  
\lim_{\varepsilon \to 0^+} \frac {1}{\varepsilon^N} \int_{\{ y \in
B(x,\varepsilon) : (y-x)\cdot\nu > 0 \}} | u(y) - a| \, dy = 0,
\qquad
\lim_{\varepsilon \to 0^+} \frac {1}{\varepsilon^N} \int_{\{ y \in
B(x,\varepsilon) : (y-x)\cdot\nu < 0 \}} | u(y) - b| \, dy = 0.
\end{equation*}
The triple $(a,b,\nu)$ is uniquely determined by the conditions above  
up to a permutation of $(a,b)$ and a change of sign of $\nu$
and is denoted by $(u^+ (x),u^- (x),\nu_u (x)).$ The jump of $u$ at $x$ is defined by
$[u](x) : = u^+(x) - u^-(x).$

\smallskip

We recall that a function $u\in L^{1}(\Omega;{\mathbb{R}}^{d})$ is said to be of bounded variation, and we write $u\in BV(\Omega;{\mathbb{R}}^{d})$, if
all its first order distributional derivatives $D_{j}u_{i}$ belong to 
$\mathcal{M}(\Omega)$ for $1\leq i\leq d$ and $1\leq j\leq N$.

The matrix-valued measure whose entries are $D_{j}u_{i}$ is denoted by $Du$
and $|Du|$ stands for its total variation.
The space $BV(\O ; {\mathbb{R}}^d)$ is a Banach space when endowed with the norm 
\begin{equation*}
\|u\|_{BV(\O ; {\mathbb{R}}^d)} = \|u\|_{L^1(\O ; {\mathbb{R}}^d)} + |Du|(\O )
\end{equation*}
and we observe that if $u\in BV(\Omega;\mathbb{R}^{d})$ then $u\mapsto|Du|(\Omega)$ is lower semicontinuous in $BV(\Omega;\mathbb{R}^{d})$ with respect to the
$L_{\mathrm{loc}}^{1}(\Omega;\mathbb{R}^{d})$ topology.

By the Lebesgue Decomposition Theorem, $Du$ can be split into the sum of two
mutually singular measures $D^{a}u$ and $D^{s}u$, the absolutely continuous
part and the singular part, respectively, of $Du$ with respect to the
Lebesgue measure $\mathcal{L}^N$. By $\nabla u$ we denote the 
Radon-Nikod\'{y}m derivative of $D^{a}u$ with respect to $\mathcal{L}^N$, so that we
can write 
\begin{equation*}
Du= \nabla u \mathcal{L}^N \lfloor \O + D^{s}u.
\end{equation*}

If $u \in BV(\O )$ it is well known that $S_u$ is countably $(N-1)$-rectifiable, see \cite{AFP},  
and the following decomposition holds 
\begin{equation*}
Du= \nabla u \mathcal{L}^N \lfloor \O + [u] \otimes \nu_u {\mathcal{H}}^{N-1}\lfloor S_u + D^cu,
\end{equation*}
\noindent where $D^cu$ is the Cantor part of the
measure $Du$.

If $\Omega$ is an open and bounded set with Lipschitz boundary then the
outer unit normal to $\partial \Omega$ (denoted by $\nu$) exists ${\mathcal{H}}^{N-1}$-a.e. and the trace for functions in $BV(\Omega;{\mathbb{R}}^d)$ is
defined.

\begin{theorem}
(Approximate Differentiability)\label{approxdiff} 
If $u \in BV(\Omega; {\mathbb{R}}^d),$ then for $\mathcal{L}^N$-a.e. $x \in\Omega$ 
\begin{equation}\label{apdif}
\lim_{\varepsilon \rightarrow 0^+} \frac{1}{\varepsilon^{N+1} }
\int_{Q(x, \varepsilon)} |u(y) - u(x) - \nabla u(x).(y-x)|\, dy  = 0.
\end{equation}
\end{theorem}

\begin{definition}
	\label{Setsoffiniteperimeter} Let $E$ be an $\mathcal{L}^{N}$- measurable
	subset of $\mathbb{R}^{N}$. For any open set $\Omega\subset\mathbb{R}^{N}$ the
	{\em perimeter} of $E$ in $\Omega$, denoted by $P(E;\Omega)$, is 
	given by
	\begin{equation}
	\label{perimeter}P(E;\Omega):=\sup\left\{  \int_{E} \mathrm{div}\varphi(x) \,dx:
	\varphi\in C^{1}_{c}(\Omega;\mathbb{R}^{d}), \|\varphi\|_{L^{\infty}}%
	\leq1\right\}  .
	\end{equation}
	We say that $E$ is a {\em set of finite perimeter} in $\Omega$ if $P(E;\Omega) <+
	\infty.$
\end{definition}

Recalling that if $\mathcal{L}^{N}(E \cap\Omega)$ is finite, then $\chi_{E}
\in L^{1}(\Omega)$, by \cite[Proposition 3.6]{AFP}, it follows
that $E$ has finite perimeter in $\Omega$ if and only if $\chi_{E} \in
BV(\Omega)$ and $P(E;\Omega)$ coincides with $|D\chi_{E}|(\Omega)$, the total
variation in $\Omega$ of the distributional derivative of $\chi_{E}$.
Moreover,  a
generalised Gauss-Green formula holds:
\begin{equation}\nonumber
{\int_{E}\mathrm{div}\varphi(x) \, dx
=\int_{\Omega}\left\langle\nu_{E}(x),\varphi(x)\right\rangle \, d|D\chi_{E}|,
\;\;\forall\,\varphi\in C_{c}^{1}(\Omega;\mathbb{R}^{d})},
\end{equation}
where $D\chi_{E}=\nu_{E}|D\chi_{E}|$ is the polar decomposition of $D\chi_{E}$.

The following approximation result can be found in \cite{BA}.

\begin{lemma}\label{polyhedra}
Let $E$ be a set of finite perimeter in $\Omega$. Then, there exists a sequence of polyhedra $E_n$, with characteristic functions $\chi_n$, such that $\chi_n\to \chi$ in $L^1(\Omega;\{0,1\})$ and $P (E_n;\Omega)\to  P(E;\Omega)$.
\end{lemma}

\subsection{BD and LD Functions}\label{BD}

We now recall some facts about functions of bounded deformation. More details can be found in \cite{ACDM, BFT, BCDM, T, TS}.

\smallskip

A function $u\in L^{1}(\Omega;{\mathbb{R}}^{N})$ is said to be of bounded deformation, and we write $u\in BD(\Omega)$, if the symmetric part of its distributional derivative $Du$, $Eu : = \dfrac{Du + Du^T}{2},$ is a matrix-valued bounded Radon measure. 
The space $BD(\O)$ is a Banach space when endowed with the norm 
\begin{equation*}
\|u\|_{BD(\Omega)} = \|u\|_{L^1(\Omega; \Rb^N)} + |Eu|(\O ).
\end{equation*}
We denote by $LD(\O)$ the subspace of $BD(\O)$ comprised of functions $u$ such that
$Eu \in L^1(\O;\Rb^{N\times N}_s)$, a counterexample due to Ornstein \cite{O} shows that $W^{1,1}(\O;\mathbb R^N) \subsetneq LD(\O)$.

The intermediate topology in the space $BD(\O)$ is the one determined by the distance
$$d(u,v) := \|u-v\|_{L^1(\O;\Rb^N)} + \big| |Eu|(\O) - |Ev |(\O)\big|, 
\;u, v \in BD(\O).$$
Hence, a sequence $\{u_n\} \subset BD(\O)$ converges to a function $u \in BD(\O)$ with respect to this topology, written $u_n \stackrel{i}{\to}u$, if and only if,
$u_n \to u$ in $L^1(\O;\Rb^N)$, $Eu_n \wsto Eu$ in the sense of measures and
$|Eu_n|(\O) \to |Eu|(\O)$. 

Recall that if $u_n \to u$ in $L^1(\Omega;\Rb^N)$ and there exists $C > 0$ such that $|Eu_n|(\Omega) \leq C, \forall n \in \Nb$, then $u \in BD(\Omega)$ and 
\begin{equation}\label{Elsc}
|Eu|(\Omega) \leq \liminf_{n\to +\infty}|Eu_n|(\Omega).
\end{equation}

By the Lebesgue Decomposition Theorem, $Eu$ can be split into the sum of two
mutually singular measures $E^{a}u$ and $E^{s}u$, the absolutely continuous
part and the singular part, respectively, of $Eu$ with respect to the
Lebesgue measure $\mathcal{L}^N$. The 
Radon-Nikod\'{y}m derivative of $E^{a}u$ with respect to $\mathcal{L}^N$, is denoted by $\E u$ so we have
\begin{equation*}
Eu= \E u \, \mathcal{L}^N \lfloor \O + E^{s}u.
\end{equation*}
With these notations we may write 
$$LD(\Omega):=\{u \in BD(\Omega): E^s u=0\}$$
and (cf. \cite{T}) $LD(\Omega)$ is a Banach space when endowed with the norm
$$\|u\|_{LD(\Omega)} : = \|u\|_{L^1(\Omega;\mathbb R^N)} + 
\|{\cal E} u\|_{L^1(\Omega;\mathbb R^N)}.$$

If $\O$ is a bounded, open subset of $\Rb^N$ with Lipschitz boundary $\Gamma$, then
there exists a linear, surjective and continuous, both with respect to the norm and to the intermediate topologies, trace operator
$$\text{tr }: BD(\O) \to L^1(\O;\Rb^N)$$
such that tr $u = u$ if $u \in BD(\O) \cap C(\overline{\O};\Rb^N)$. Furthermore, the following Gauss-Green formula holds
\begin{equation}\label{GGBD}
\int_{\O}(u \odot D \varphi)(x) \, dx + \int_{\O}\varphi(x) \, dEu(x) =
\int_{\Gamma}\varphi (x)(\text{tr}u \odot \nu)(x) \, d \mathcal H ^{N-1}(x),
\end{equation}
for every $\varphi \in C^1(\overline{\O})$ (cf. \cite{ACDM, T}).

\smallskip

The following lemma is proved in \cite{BFT}.

\begin{lemma}\label{lemma2.2BFT}
Let $u \in BD(\O)$ and let $\rho \in C_0^{\infty}(\Rb^N)$ be a non-negative function such that ${\rm supp}(\rho) \subset\subset B(0,1)$, $\rho(-x) = \rho(x)$ for every $x \in \Rb^N$ and $\displaystyle \int_{\Rb^N}\rho(x) \, dx = 1$. For any $n \in \Nb$ set $\rho_n(x) : = n^N\rho(nx)$ and
$$u_n(x) := (u *\rho_n)(x) = \int_{\O}u(y)\rho_n(x-y) \, dy, \; \;
\mbox{for} \; x \in \left\{y \in \O : {\rm dist}(y,\partial \O) > \frac{1}{n}\right\}.$$
Then $u_n \in C^{\infty}\left(\left\{y \in \O : {\rm dist}(y,\partial \O) > \frac{1}{n}\right\};\Rb^N\right)$ and
\begin{itemize}
\item[i)] for any non-negative Borel function $h : \O \to \Rb$
$$\int_{B(x_0,\varepsilon)}h(x) |\E u_n(x)| \, dx \leq 
\int_{B(x_0,\varepsilon+ \frac{1}{n})}(h*\rho_n)(x) \, d|E u|(x),$$
whenever $\e + \frac{1}{n} < {\rm dist}(x_0,\partial \O)$;
\item[ii)] for any positively homogeneous of degree one, convex function 
$\theta : \Rb^{N\times N}_{\rm sym} \to [0,+\infty[$ and any $\e \in \, ]0,{\rm dist}(x_0,\partial \O)[$ such that $|E u|(\partial B(x_0,\e)) = 0$,
$$\lim_{n\rightarrow +\infty}\int_{B(x_0,\varepsilon)} \theta(\E u_n(x)) \, dx =
\int_{B(x_0,\varepsilon)}\theta\left(\frac{d Eu}{d |Eu|}\right) \, d |Eu|,$$
\item[iii)] $\displaystyle \lim_{n \to + \infty}u_n(x) = \widetilde u(x)$ and
$\displaystyle \lim_{n \to + \infty}(|u_n -u| * \rho_n)(x) = 0$ for every
$x \in \O \setminus S_u$, whenever $u \in L^{\infty}(\O;\Rb^N)$.
\end{itemize}
\end{lemma}

The following result, proved in \cite{T}, see also \cite[Theorem 2.6]{BFT}, shows that it is possible to approximate any $BD(\O)$ function $u$ by a sequence of smooth functions which preserve the trace of $u$.

\begin{theorem}\label{densitysmooth}
Let $\O$ be a bounded, connected, open set with Lipschitz boundary. For every $u \in BD(\O)$, there exists a sequence of smooth functions $\{u_n\} \subset C^{\infty}(\O;\Rb^N) \cap W^{1,1}(\O;\Rb^N)$ such that $u_n \stackrel{i}{\to}u$ and
tr $u_n =$ tr $u$. If, in addition, $u \in LD(\O)$, then $\E u_n \to \E u$ in
$L^1(\O;\Rb_s^{N\times N}).$ 
\end{theorem} 

It is also shown in \cite{T} that if $\O$ is an open, bounded subset of $\Rb^N$, with Lipschitz boundary, then $BD(\O)$ is compactly embedded in $L^q(\O;\Rb^N)$, for every $1 \leq q < \frac{N}{N-1}.$ 
In particular, 
the following result holds.

\begin{theorem}\label{THM2.10BFT}
Let $\Omega$ be an open, bounded subset of $\mathbb R^N$, with Lipschitz boundary
and let $1\leq q <\frac{N}{N-1}$. If $\{u_n\}$ is bounded in $BD(\Omega)$, then there exist $u \in BD(\Omega)$ and a subsequence $\{ u_{n_k}\}$ of  $\{u_n\}$ such that $u_{n_k}\to u$ in $L^q (\Omega;\mathbb R^N)$.
\end{theorem}

\smallskip

If $u \in BD(\O )$ then $J_u$ is countably $(N-1)$-rectifiable, see \cite{ACDM}, 
and the following decomposition holds 
\begin{equation*}
Eu= \E u \mathcal{L}^N \lfloor \O + [u] \odot \nu_u {\mathcal{H}}^{N-1}\lfloor J_u + E^cu,
\end{equation*}
\noindent where $[u] = u^+ - u^-$, $u^{\pm}$ are the traces of $u$ on the sides of $J_u$ determined by the unit normal $\nu_u$ to $J_u$ and $E^cu$ is the Cantor part of the measure $Eu$ which vanishes on  Borel sets $B$ with 
$\mathcal H^{N-1}(B) < + \infty.$

We end this subsection by pointing out that the equivalent of \eqref{apdif}, with $\E u(x)$ replacing $\nabla u(x)$, is false (see \cite{ACDM}). 
However the following result holds (cf. \cite[Theorem 4.3]{ACDM} and \cite[Theorem 2.5]{E2}).

\begin{theorem}
(Approximate Symmetric Differentiability)\label{approxsymdiff} 
If $u \in BD(\Omega),$ then, for $\mathcal{L}^N$-a.e. $x \in\Omega$ ,
there exists an $N\times N$ matrix $\nabla u(x)$ such that
\begin{equation}\label{apdifBD}
\lim_{\varepsilon \rightarrow 0^+} \frac{1}{\varepsilon^{N+1} }
\int_{B_\varepsilon(x)} |u(y) - u(x) - \nabla u(x).(y-x)|\, dy  = 0,
\end{equation}
\begin{equation}\label{apsymdif}
\lim_{\varepsilon \rightarrow 0^+} \frac{1}{\varepsilon^{N} }
\int_{B_\varepsilon(x)} \frac{|\langle u(y) - u(x) - \E u(x).(y-x), y-x\rangle|}{|y-x|^2}\, dy  = 0,
\end{equation}
for $\mathcal L^N$- a.e. $x \in \Omega$.
Furthermore
$${\mathcal L}^N(\{x\in \Omega:|\nabla u(x)|>t \})\leq 
\frac {C(N,\Omega)}{t}\|u\|_{BD(\Omega)}, \;\;\forall t>0,$$
with $C(N,\Omega)>0$ depending only on $N$ and $\Omega$. 
\end{theorem}
From \eqref{apdifBD} and \eqref{apsymdif} it follows that
$\displaystyle \E u=\frac{\nabla u+\nabla u^T}{2}$.

\medskip

We denote by $\mathcal R$ the kernel of the linear operator $E$ consisting of the class of rigid motions in $\mathbb R^N$, i.e., affine maps of the form $Mx + b$ where $M$ is a skew-symmetric $N \times N$ matrix and $b\in \mathbb R^N$. 
$\mathcal R$ is therefore closed and finite-dimensional so
it is possible to define the orthogonal projection $P : BD(\Omega)\to \mathcal R$. This operator belongs to the class considered in the following Poincar\'e-Friedrichs
type inequality for $BD$ functions (see \cite{ACDM}, \cite{K} and \cite{T}).

\begin{theorem}\label{Thm2.8BFT}
Let $\Omega$ be a bounded, connected, open subset of $\mathbb R^N$, with Lipschitz boundary, and let $R : BD(\Omega)\to \mathcal R$ be a continuous linear map which leaves the elements of $\mathcal R$ fixed. Then there exists a constant 
$C(\Omega, R)$ such that
\begin{align*}
\int_\Omega |u(x) - R(u)(x)| \, dx \leq C(\Omega,R) \, |E u|(\O), \;  \mbox{ for every } u\in BD(\Omega).
\end{align*}
\end{theorem}

\subsection{Notions of Quasiconvexity}\label{defqcx}

\begin{definition}
[\cite{BFT}, Definition 3.1]\label{BFTsqcx}
A Borel measurable 
function $f:\mathbb R^{N\times N}_s\to \mathbb R$ is said to be 
{\em symmetric quasiconvex} if
\begin{equation}\label{sqcx}
f(\xi)\leq\int_Q 
f(\xi+\E \varphi(x)) \, dx,
\end{equation}
for every $\xi \in \mathbb R^{N\times N}_s$ and for 
every $\varphi \in C^{\infty}_{\rm per}(Q;\mathbb R^N)$.  
\end{definition}

\begin{remark}\label{LDper}
{\rm The above property \eqref{sqcx} is independent of the size, orientation and centre of the cube over which the integration is performed. Also, if $f$ is upper semicontinuous and locally bounded from above, using Fatou's Lemma and the density of smooth functions in $LD(Q)$, it follows that in \eqref{sqcx} $C^{\infty}_{\rm per}(Q;\mathbb R^N)$ may be replaced by $LD_{\rm per}(Q).$}
\end{remark}

Given $f:\mathbb R^{N\times N}_s\to \mathbb R$, the symmetric quasiconvex envelope of $f$, $SQf$, is defined by
\begin{equation}\label{sqcxenv}
SQf(\xi) : = \inf \bigg\{\int_Q f(\xi+\E \varphi(x)) \, dx : \varphi \in C^{\infty}_{\rm per}(Q;\mathbb R^N)\bigg\}.
\end{equation}
It is possible to show that $SQf$ is the greatest symmetric quasiconvex function that is less than or equal to $f$. Moreover, definition \eqref{sqcxenv} is independent of the domain, i.e.
\begin{align}\label{sqcx0}SQf(\xi) : = \inf \bigg\{\frac{1}{\L ^N(D)}\int_{D} f(\xi+\E \varphi(x)) \, dx : \varphi \in C^{\infty}_{0}(D;\mathbb R^N)\bigg\}
\end{align}
whenever $D \subset \Rb^N$ is an open, bounded set with $\L ^N(\partial D) = 0$.

In \cite{E}, a Borel measurable function $f:\mathbb R^{N\times N}_s\to \mathbb R$ is said to be symmetric quasiconvex if and only if 
\begin{align}\label{SQfE}
f(\xi)\leq \frac{1}{\mathcal L^N(D)}\int_D f(\xi +\E\varphi(x)) \, dx 
\hbox{ for all }\varphi \in W^{1,\infty}_0(D;\mathbb R^N),
\end{align}
and it is stated that $f$ is symmetric quasiconvex if and only if $f\circ \pi$ is quasiconvex in the sense of Morrey, 
where $\pi$ is the projection of $\mathbb R^{N \times N}$ onto 
$\mathbb R^{N\times N}_{s}$. 

Let us show that these two notions coincide. Observe first that, for any 
$\varphi \in C^\infty_0(D;\mathbb R^N)$,
\begin{align}\label{SQfestW11}
SQf(\xi)\leq \frac{1}{\mathcal L^N(D)}\int_D SQf(\xi+ \E \varphi(x))dx=
\frac{1}{\mathcal L^N(D)}\int_D (SQf \circ \pi)(\xi+ \nabla\varphi(x))dx.
\end{align} 
If $f$ is upper semicontinuous and satisfies a growth condition from above as in \eqref{growthint}, then $SQf$ in \eqref{sqcx0} is symmetric quasiconvex also in the sense of \cite{E}. Indeed, $SQf$ satisfies the same growth condition \eqref{growthint} and a density argument as in \cite{BM} shows that 
$SQf \circ \pi$ is $W^{1,1}$-quasiconvex, hence $W^{1,\infty}$-quasiconvex, i.e., $\varphi$ can be chosen in $W^{1,\infty}_0(D;\mathbb R^N)$. Thus,
\begin{align}\label{SQfisE}
SQf(\xi)\leq \frac{1}{\mathcal L^N(D)}\int_D SQf(\xi+ \E \varphi(x)) \, dx
\leq \frac{1}{\mathcal L^N(D)}\int_D f(\xi+ \E \varphi(x)) \, dx,
\end{align}
for every $\varphi \in W^{1,\infty}_0(D;\mathbb R^N)$.
Therefore, denoting by $SQf_E$ the symmetric quasiconvexification 
\begin{align}\label{symqcxE}
SQf_E(\xi) : = \inf \bigg\{\frac{1}{\L ^N(D)}\int_{D} f(\xi+\E \varphi(x)) \, dx : \varphi \in W^{1,\infty}_{0}(D;\mathbb R^N)\bigg\},
\end{align}
and by $SQf$ the symmetric quasiconvexification defined through \eqref{sqcx0},
trivially $SQf_E\leq SQf$ and by \eqref{SQfisE} we have equality.

Actually, under linear growth conditions and upper semicontinuity of $f$, we may
also conclude that
\begin{align*}
SQf_E(\xi) : = \inf \bigg\{\frac{1}{\L^N(D)}\int_{D} f(\xi+\E \varphi(x)) \, dx : \varphi \in W^{1,1}_{0}(D;\mathbb R^N)\bigg\}.
\end{align*} 

\section{Auxiliary Results}\label{auxres}

We recall that for $u \in BD(\Omega)$ and $\chi \in BV(\Omega;\{0,1\})$ the energy under consideration is
\begin{equation}\label{F}
F(\chi,u;\Omega) : = 
\displaystyle \int_{\Omega}\chi(x) W_1({\cal E} u(x)) 
+ (1 - \chi(x))W_0(\E u(x)) \, dx + |D\chi|(\Omega), 
\end{equation}
and our aim is to obtain an integral representation for the localised relaxed functionals, defined for $A \in \mathcal O(\Omega)$, by
\begin{align}\label{calF}
\mathcal{F}\left(\chi,u;A\right)   &  
:=\inf\bigg\{ \liminf_{n\rightarrow +\infty} F(\chi_n,u_n;A): 
u_{n} \in W^{1,1}(A;\mathbb{R}^{N}),
\chi_{n} \in BV(A;\{0,1\}), \\
&  \hspace{4cm} 
u_{n}\to u\text{ in }
L^{1}(A;\mathbb{R}^{N}),
\chi_{n}\to\chi\text{ in }
L^1(A;\{0,1\})  \bigg\},\nonumber
\end{align}
\begin{align}\label{calFLD}
\mathcal{F}_{LD}\left(\chi,u;A\right)   &  
:=\inf\bigg\{ \liminf_{n\rightarrow +\infty} F(\chi_n,u_n;A): 
u_{n} \in LD(A), \chi_{n} \in BV(A;\{0,1\}), \\
&  \hspace{4cm} 
u_{n}\to u\text{ in }
L^{1}(A;\mathbb{R}^{N}),
\chi_{n}\to\chi\text{ in }
L^1(A;\{0,1\})  \bigg\},\nonumber
\end{align}
where the densities $W_i$, $i = 0,1$, are continuous functions such 
that 
\begin{equation}\label{growth}
\exists \, \alpha, \beta > 0 \text{ such that }
\alpha |\xi| \leq W_i(\xi) \leq \beta (1 + |\xi|), \; \; \forall 
\xi \in \Rb^{N \times N}_s
\end{equation}
and where, for purposes of notation, we let 
$f: \{0,1\} \times \mathbb R^{N \times N}_s \to [0,+\infty)$
be defined as
\begin{equation}\label{density}
f\left( q,\xi\right)  :=q  W_1(\xi)+ (1-q)W_0(\xi). 
\end{equation}

It follows from the definition of the recession function \eqref{recS} and from the growth conditions \eqref{growth} that for every $q \in \{0,1\}$ and every 
$\xi \in \Rb^{N \times N}_s$
\begin{equation}\label{finftygr}
\alpha |\xi| \leq f^{\infty}(q,\xi) \leq \beta |\xi|.
\end{equation}

It is an immediate consequence of \eqref{growth} that 
\begin{equation}\label{G}
|f(q_1,\xi) - f(q_2,\xi)| \leq \beta \, |q_1 - q_2|(1 + |\xi|), \; \;
\forall q_1, q_2 \in \{0,1\},
\forall \xi \in \Rb^{N \times N}_s,
\end{equation}
from which it follows that 
\begin{equation}\label{Gfinfty}
|f^{\infty}(q_1,\xi) - f^{\infty}(q_2,\xi)| \leq \beta \, |q_1 - q_2| \,|\xi|, \; \;
\forall q_1, q_2 \in \{0,1\},
\forall \xi \in \Rb^{N \times N}_s.
\end{equation}

The following additional hypothesis will be used to write the density of the jump term in the form given in \eqref{KSQ}
\begin{equation}\label{finfty}
\exists \, 0 < \gamma \leq 1, \exists \, C, L > 0 : \; t\,|\xi| > L \Rightarrow 
\left| f^{\infty}(q,\xi) - \frac{f(q,t\xi)}{t}\right| 
\leq C \frac{|\xi|^{1-\gamma}}{t^\gamma},
\end{equation}
for every $q \in \{0,1\}$ and every $\xi \in \Rb^{N \times N}_s$.
As pointed out in \cite{FM}, this can be stated equivalently as
\begin{equation}\label{finfty2}
\exists \, 0 < \gamma \leq 1, \exists \, C > 0 \mbox{ such that }
\left| f^{\infty}(q,\xi) - f(q,\xi)\right| 
\leq C \left(1 + |\xi|^{1-\gamma}\right),
\end{equation}
for every $q \in \{0,1\}$ and every $\xi \in \Rb^{N \times N}_s$.

Under our assumed growth conditions \eqref{growth}, we observe that if $f$ satisfies \eqref{finfty}, or equivalently \eqref{finfty2}, then the same holds for its symmetric quasiconvex envelope $SQf$.
To this end, we recall that, under the hypothesis \eqref{growth}, the
recession function of a symmetric quasiconvex function is still symmetric quasiconvex (see \cite[Remarks 8 and 9]{R}) and we begin by stating the following results (cf. \cite[(iv) and (v) in Remark 3.2]{CZ} and \cite[Propositions 2.6, 2.7]{RZ} for the quasiconvex counterpart).

\begin{proposition}\label{SQfinfty=}
Let $f:\{0,1\} \times \mathbb R^{N \times N}_s\to [0, + \infty)$ be a continuous function as in \eqref{density} and satisfying \eqref{growth} and \eqref{finfty}. Let $f^\infty$ and $SQf$ be its recession function and its symmetric quasiconvex envelope, defined by \eqref{recS} and \eqref{sqcxenv}, respectively.
Then
\begin{equation}\label{QWinf=}
SQ(f^\infty)(q,\xi)= (SQf)^\infty(q,\xi) \;\;\;\hbox{ for every }(q,\xi) \in \{0,1\} \times \mathbb R^{N \times N}_s.
\end{equation}
\end{proposition}

\begin{proposition}\label{propperH5}	
Let $f:\{0,1\} \times \mathbb R^{N \times N}_s\to [0, + \infty)$ be a continuous function as in \eqref{density}, satisfying \eqref{growth} and \eqref{finfty}. 
Then, there exist $\gamma\in [0,1)$ and $C>0$ such that
$$ \displaystyle\left| (SQf)^\infty(q,\xi)- SQf(q, \xi)\right|
\leq C\big( 1+|\xi|^{1-\gamma}\big), \;\;
\forall\ (q,\xi)\in \{0,1\} \times \mathbb R^{N \times N}_s.$$
\end{proposition}

The growth conditions \eqref{growth}, as well as standard diagonalisation arguments, allow us to prove the following properties of the functional 
$\mathcal F (\chi,u;A)$ defined in \eqref{calF}.

\begin{proposition}\label{firstprop}
Let $A \in \mathcal O(\Omega)$, $u \in BD(A)$, $\chi \in BV(A;\{0,1\})$ and 
$F(\chi,u;A)$ be given by \eqref{F}. If $W_i$, $i = 0,1$, satisfy \eqref{growth}, then
\begin{itemize}
\item[i)] there exists $C > 0$ such that 
$$C\left(|Eu|(A) + |D\chi|(A)\right) \leq \mathcal F (\chi,u;A) \leq 
C\left(\L^N(A) + |Eu|(A) + |D\chi|(A)\right);$$
\item[ii)] $\mathcal F (\chi,u;A)$ is always attained, that is, there exist sequences
$\{u_n\} \subset  W^{1,1}(A;\mathbb{R}^{N})$ and
$\{\chi_{n}\} \subset BV(A;\{0,1\})$ such that $u_{n}\to u$ in $L^{1}(A;\mathbb{R}^{N})$,
$\chi_{n}\to\chi$ in $L^1(A;\{0,1\})$ and
$$\mathcal F(\chi,u;A) = \lim_{n\rightarrow\infty}F(\chi_n,u_n;A);$$
\item[iii)] if $\{u_n\} \subset  W^{1,1}(A;\mathbb{R}^{N})$ and
$\{\chi_{n}\} \subset BV(A;\{0,1\})$ are such that $u_{n}\to u$ in $L^{1}(A;\mathbb{R}^{N})$ and 
$\chi_{n}\to\chi$ in $L^1(A;\{0,1\})$, then
$$\mathcal F(\chi,u;A) \leq \liminf_{n\to +\infty}\mathcal F(\chi_n, u_n;A).$$
\end{itemize}
\end{proposition}
\begin{proof}
$i)$ The upper bound follows from the growth condition from above of $W_i$, 
$i  =0,1$ and by fixing $\chi_n = \chi$ as a test sequence for 
$\mathcal F (\chi,u;A)$, whereas the lower bound is a consequence of the inequality from below in \eqref{growth}, \eqref{Elsc} and the lower semicontinuity of the total variation of Radon measures.

The conclusions in $ii)$ and $iii)$ follow by standard diagonalisation arguments.
\end{proof}

\begin{remark}{\rm Analogous conclusions also hold for the functional 
$\mathcal F_{LD} (\chi,u;A)$.}
\end{remark}

\begin{remark}\label{remkW11LD}{\rm 
Assuming that the continuous functions $W_0$ and $W_1$ satisfy the growth hypothesis \eqref{growth}, it follows from the density of smooth functions in $LD(\Omega)$ and a diagonalisation argument that 
$$\mathcal{F}\left(\chi,u;A\right) = \mathcal{F}_{LD}\left(\chi,u;A\right), \;
\mbox{for every} \; \chi \in BV(A;\left\{0,1\right\}), u \in BD(\Omega),
A \in {\cal O}(\Omega).$$}
\end{remark}
\begin{proof}
As $W^{1,1}(A;\mathbb R^N) \subset LD(A)$, one inequality is trivial. In order to show the reverse one, let $\{u_n\} \subset LD(A)$, 
$\{\chi_n\} \subset BV(A;\left\{0,1\right\})$ be such that 
$u_{n}\to u$ in $L^{1}(A;\mathbb{R}^{N})$,
$\chi_{n}\to\chi$ in $L^1(A;\{0,1\})$ and
$$\mathcal{F}_{LD}\left(\chi,u;A\right) = \lim_n \left[
\int_{A}\chi_n(x) W_1({\cal E} u_n(x)) + (1 - \chi_n(x))W_0(\E u_n(x)) \, dx + |D\chi_n|(A) \right].$$
By Theorem \ref{densitysmooth}, for each $n \in \Nb$, let $v_{n,k} \in W^{1,1}(A;\mathbb R^N)$ be such that
$v_{n,k} \to u_n$ in $L^{1}(A;\mathbb{R}^{N})$, as $k \to + \infty$, and
$\E v_{n,k} \to \E u_n$ in $L^{1}(A;\mathbb{R}^{N\times N}_s)$, as $k \to + \infty$.
By passing to a subsequence, if necessary, assume also that $\displaystyle \lim_{k\rightarrow\infty}\E v_{n,k}(x) = \E u_n(x)$, for a.e. $x \in A$.
By \eqref{growth} and Fatou's Lemma we obtain
$$
\int_A\chi_n(x)\left[C(1 + |\E u_n(x)|) - W_1(\E u_n(x))\right] dx \leq
\liminf_{k\to +\infty}\int_A\chi_n(x)
\left[C(1 + |\E v_{n,k}(x)|) - W_1(\E  v_{n,k}(x))\right] dx $$
so that 
$$
\int_A\chi_n(x)W_1(\E u_n(x)) \, dx \geq
\limsup_{k\to +\infty}\int_A\chi_n(x)W_1(\E  v_{n,k}(x)) \, dx, $$
and likewise for the term involving $(1 - \chi_n)W_0$. From the previous inequalities
we conclude that
$$F(\chi_n,u_n;A) \geq \limsup_{k\to +\infty}F(\chi_n,v_{n,k};A).$$
Since $v_{n,k} \to u_n$ in $L^{1}(A;\mathbb{R}^{N})$, as $k \to + \infty$, and 
$u_{n}\to u$ in $L^{1}(A;\mathbb{R}^{N})$, by a diagonalisation argument there exists
a sequence $k_n \to + \infty$ such that $v_{n,k_n} \to u$ in $L^{1}(A;\mathbb{R}^{N})$ and 
$$F(\chi_n,v_{n,k_n};A) \leq F(\chi_n,u_n;A) + \frac{1}{k_n}.$$
As $\{\chi_n\}$, $\{v_{n,k_n}\}$ are admissible for $\mathcal{F}\left(\chi,u;A\right)$ it follows that
\begin{align*}
\mathcal{F}\left(\chi,u;A\right) \leq \liminf_{n\to +\infty}F(\chi_n,v_{n,k_n};A)
\leq \limsup_{n\to +\infty} \left(F(\chi_n,u_n;A) + \frac{1}{k_n}\right) = 
\mathcal{F}_{LD}\left(\chi,u;A\right).
\end{align*}
\end{proof}

A straightforward adaptation of the proof of 
\cite[Proposition 3.7]{BFT} yields the following result which enables us to prove the nested subadditivity property of the functional $\mathcal{F}\left(\chi,u;\cdot\right)$.

\begin{proposition}\label{slicing}
Let $A \in \mathcal O (\Omega)$ and assume that $W_0, W_1$ satisfy the growth condition \eqref{growth}. Let $\{\chi_n\} \subset BV(A;\{0,1\})$ and
$\{u_n\}, \{v_n\} \subset BD(A;\mathbb R^N)$ be sequences
satisfying $u_{n} - v_n \to 0$ in $L^{1}(A;\mathbb{R}^{N})$,
$\sup_n |Eu_n|(A) < + \infty$, $|Ev_n| \overset{\ast}{\rightharpoonup} \mu$ and
$|Ev_n| \to \mu(A)$. Then there exist subsequences $\{v_{n_k}\}$ of $\{v_n\}$, 
$\{\chi_{n_k}\}$ of $\{\chi_n\}$ and there exists a sequence $\{w_k\} \subset BD(A)$
such that $w_k = v_{n_k}$ near $\partial A$, 
$w_k - v_{n_k} \to 0$ in $L^{1}(A;\mathbb{R}^{N})$ and
$$\limsup_{k\to +\infty}F(\chi_{n_k},w_k;A) \leq 
\liminf_{n\to +\infty}F(\chi_n,u_n;A).$$
\end{proposition}

It is clear from the proof that if the original sequences $\{u_n\}, \{v_n\}$ belong
to $W^{1,1}(A;\mathbb{R}^{N})$ then the sequence $\{w_k\}$ will also be in this space.

\begin{proposition}\label{nestedsa}
Assume that $W_0$ and $W_1$ are continuous functions satisfying \eqref{growth}. 
Let $u \in BD(\Omega)$, $\chi \in BV(\Omega;\{0,1\})$ and $S, U, V \in \mathcal O(\Omega)$ be such that $S \subset \subset V \subset U.$ Then
$$\mathcal{F}\left(\chi,u;U\right) \leq \mathcal{F}\left(\chi,u;V\right) 
+ \mathcal{F}\left(\chi,u;U\setminus \overline S\right).$$
\end{proposition}
\begin{proof}
By Proposition \ref{firstprop}, $ii)$, let 
$\{v_n\} \subset W^{1,1}(V;\mathbb{R}^{N})$, 
$\{w_n\} \subset W^{1,1}(U\setminus \overline S;\mathbb{R}^{N})$, 
$\{\chi_n\} \subset BV(V;\{0,1\})$ and
$\{\theta_n\} \subset BV(U\setminus \overline S;\{0,1\})$ be such that
$v_n \to u$ in $L^{1}(V;\mathbb{R}^{N})$,
$w_n \to u$ in $L^{1}(U\setminus \overline S;\mathbb{R}^{N})$,
$\chi_{n}\to\chi$ in $L^1(V;\{0,1\})$
$\theta_{n}\to\chi$ in 
$L^1(U\setminus \overline S;\{0,1\})$ and
\begin{align}\label{VUS1}
\mathcal{F}\left(\chi,u;V\right) & = \lim_{n\rightarrow\infty}F(\chi_n,v_n;V) \\
\mathcal{F}\left(\chi,u;U\setminus \overline S\right) & = \lim_{n\rightarrow\infty}F(\theta_n,w_n;U\setminus \overline S). \label{VUS2}
\end{align}
Let $V_0 \in \mathcal O_{\infty}(\O)$ satisfy 
$S \subset \subset V_0 \subset \subset V$ and $|E u| (\partial V_0) = 0$, 
$|D\chi|(\partial V_0) = 0$. Applying Proposition \ref{slicing} to $\{v_n\}$ and $u$ in $V_0$, we obtain a subsequence $\{\overline \chi_n\}$ of $\{\chi_n\}$ and a sequence $\{\overline v_n\} \subset W^{1,1}(V_0;\mathbb{R}^{N})$
such that $\overline v_n = u$ near $\partial V_0$, $\overline v_n \to u$ in 
$L^{1}(V_0;\mathbb{R}^{N})$ and
\begin{equation}\label{sl1}
\limsup_{n\to +\infty}F(\overline \chi_n,\overline v_n;V_0) \leq 
\liminf_{n\to +\infty}F(\chi_n,v_n;V_0).
\end{equation}
A further application of Proposition \ref{slicing}, this time to $\{w_n\}$ and $u$ in $U \setminus \overline V_0$, yields a subsequence $\{\overline \theta_n\}$ of $\{\theta_n\}$ and a sequence 
$\{\overline w_n\} \subset W^{1,1}(U \setminus \overline V_0;\mathbb{R}^{N})$
such that $\overline w_n = u$ near $\partial V_0$, $\overline w_n \to u$ in 
$L^{1}(U \setminus \overline V_0;\mathbb{R}^d)$ and
\begin{equation}\label{sl2}
\limsup_{n\to +\infty}F(\overline \theta_n,\overline w_n;U \setminus \overline V_0) \leq \liminf_{n\to +\infty}F(\theta_n,w_n;U \setminus \overline V_0).
\end{equation}
Define 
$$z_n : = \begin{cases}
\overline v_n,& {\rm in } \; V_0\\
\overline w_n,& {\rm in } \; U \setminus V_0,
\end{cases}$$
notice that, by the properties of $\{\overline v_n\}$ and $\{\overline w_n\}$,
$\{z_n\} \subset W^{1,1}(U;\mathbb{R}^{N})$ and 
$z_n \to u$ in $L^{1}(U;\mathbb{R}^{N})$.

We must now build a transition sequence $\{\eta_n\}$ between $\{\overline\chi_n\}$ and $\{\overline\theta_n\}$, in such a way that an upper bound 
for the total variation of $\eta_n$ is obtained.
In order to connect these functions without adding more interfaces, we argue 
as in \cite{BMMO} (see also \cite{BZ}). 
For $\delta > 0$ consider
$$
V_\delta:=\{x \in V: {\rm dist}(x,V_0) < \delta\},
$$
where $\delta$ is small enough so that $\overline w_n = u$ in 
$V_{\delta} \setminus \overline V_0$ and
\begin{equation}\label{Odelta}
\int_{V_{\delta} \setminus \overline V_0}C(1 + |u(x)|) \, dx = O(\delta).
\end{equation}
Given $x \in V$, let $d(x) := {\rm dist}(x;V_0)$. Since the distance function 
to a fixed set is Lipschitz continuous, 
applying the change of variables formula (see Theorem 2, Section 3.4.3, 
in \cite{EG}) yields
$$
\int_{V_\delta \setminus \overline{V_0}}|\overline\chi_n(x)-\overline\theta_n(x)| 
|{\rm det}\nabla d(x)| \, dx =
\int_0^\delta \left[ \int_{d^{-1}(y)} |\overline\chi_n(x)-\overline\theta_n(x)| \,
d {\mathcal H}^{N-1}(x) \right] dy$$
and, as $|{\rm det}\nabla d(x)|$ 
is bounded and $\overline\chi_n - \overline\theta_n \to 0$  in 
$L^1(V \cap (U \setminus \overline S);\{0,1\})$, it follows that, for almost every $\rho \in [0; \delta]$, we have
\begin{equation}\label{rho}
\lim_{n\to +\infty}\int_{d^{-1}(\rho)}|\overline\chi_n(x) - \overline\theta_n(x)| \, 
d{\mathcal H}^{N-1}(x) = 
\lim_{n\to +\infty}\int_{\partial V_\rho}|\overline\chi_n(x)-\overline\theta_n(x)| \,
d {\mathcal H}^{N-1}(x) = 0.
\end{equation}
Fix $\rho_0\in [0; \delta]$ such that $|D\chi|(\partial V_{\rho_0})= 0$ and \eqref{rho} holds. We
observe that $V_{\rho_0}$ is a set with locally Lipschitz boundary since 
it is a level set of a Lipschitz function (see, for example, \cite{EG}). Hence, for every $n$, we can consider $\overline\chi_n, \overline\theta_n$ on $\partial V_{\rho_0}$ in the sense of traces and define
\begin{equation*}
\eta_n :=\begin{cases}
\overline\chi_n, &\hbox{ in } V_{\rho_0}\\
\overline\theta_n, &\hbox{ in } U\setminus V_{\rho_0}.
\end{cases}
\end{equation*}
Then $\{\eta_n\} \subset BV(U;\{0,1\})$, 
$\eta_{n}\to\chi$ in $L^1(U;\{0,1\})$ and so $\{\eta_n\}$
and $\{z_n\}$ are admissible for $\mathcal{F}\left(\chi,u;U\right)$. Therefore, by
\eqref{rho}, \eqref{growth}, \eqref{sl1}, \eqref{sl2}, \eqref{Odelta}, \eqref{VUS1} and \eqref{VUS2},
\begin{align*}
\mathcal{F}\left(\chi,u;U\right) &\leq \liminf_{n\to +\infty}F(\eta_n,z_n;U) \\
& = \liminf_{n\to +\infty}\left[F(\overline\chi_n,\overline v_n; V_0)  +
\int_{V_{\rho_0}\setminus \overline V_0}\overline\chi_n(x)W_1(\E u(x)) + 
(1 - \overline \chi_n(x))W_0(\E u(x)) \, dx \right.\\ 
& \left. \hspace{2cm} + |D\overline \chi_n|(V_{\rho_0}\setminus V_0) 
 + F(\overline\theta_n,\overline w_n; U \setminus V_{\rho_0}) + \int_{\partial V_{\rho_0} }|\overline \chi_n(x) - \overline \theta_n(x)| 
\, d  {\mathcal H}^{N-1}(x)\right] \\
& \leq \limsup_{n\to +\infty}F(\overline\chi_n,\overline v_n; V_0) + 
\limsup_{n\to +\infty}F(\overline\theta_n,\overline w_n; U \setminus \overline V_{0})
+ \int_{V_{\rho_0}\setminus \overline V_0}C \big(1 + |\E u(x)|\big) \, dx \\
& \hspace{7.9cm} + \limsup_{n\to +\infty} |D\chi_n|(V_{\rho_0}\setminus V_0) \\
& \leq \liminf_{n\to +\infty}F(\chi_n,v_n; V_0) +
\liminf_{n\to +\infty}F(\theta_n,w_n; U \setminus \overline V_{0}) + O(\delta) +
\limsup_{n\to +\infty} |D\chi_n|(V_{\rho_0}\setminus V_0) \\
& \leq \limsup_{n\to +\infty}F(\chi_n,v_n; V) +
\limsup_{n\to +\infty}F(\theta_n,w_n; U \setminus \overline S) + O(\delta) \\
& = \mathcal{F}\left(\chi,u;V\right) + 
\mathcal{F}\left(\chi,u;U\setminus \overline S\right) + O(\delta)
\end{align*}
so the result follows by letting $\delta \to 0^+$.
\end{proof}

\begin{proposition}\label{traceRm}
Let $W_0$ and $W_1$ be continuous functions satisfying \eqref{growth}. 
For every $u \in BD(\Omega)$, $\chi \in BV(\Omega;\{0,1\})$, $\mathcal{F}\left(\chi,u;\cdot\right)$ is the restriction to $\mathcal O(\O)$ of a Radon measure.
\end{proposition}
\begin{proof}
By Proposition \ref{firstprop}, $ii)$, let 
$\{u_n\} \subset W^{1,1}(\O;\mathbb{R}^{N})$, 
$\{\chi_n\} \subset BV(\O;\{0,1\})$, be such that
$u_n \to u$ in $L^{1}(\O;\mathbb{R}^{N})$,
$\chi_{n}\to\chi$ in $L^1(\O;\{0,1\})$ and
$$\mathcal{F}\left(\chi,u;\O\right)  = \lim_{n\rightarrow\infty}F(\chi_n,u_n;\O).$$
Let $\mu_n = f(\chi_n(\cdot), \E u_n (\cdot)) {\cal L}^N\lfloor{\Omega}
+ |D \chi_n|$ and extend this sequence of measures outside of 
$\Omega$ by setting, for any Borel set $E \subset \mathbb R^N$,
$$\lambda_n(E) = \mu_n(E \cap \Omega).$$ 
Passing, if necessary, to a subsequence, we can assume that there exists a 
non-negative Radon measure $\mu$ (depending on $\chi$ and $u$) on 
$\overline{\Omega}$  such that $\lambda_n \wsto \mu$ in the sense of measures 
in $\overline {\Omega}$.
Let $\varphi_k \in C_0(\overline{\Omega})$ be an increasing sequence of functions 
such that $0 \leq \varphi_k \leq 1$ and $\varphi_k(x) \to 1$ a.e. in 
$\overline{\Omega}$. Then, by Fatou's Lemma and by the choice of $\{u_n\}$, $\{\chi_n\}$, we have
\begin{align*}
\mu(\overline{\Omega}) & = \int_{\overline{\Omega}}
\liminf_{k \to + \infty} \varphi_k(x) \, d\mu
 \leq \liminf_{k \to +\infty}\int_{\overline{\Omega}}\varphi_k(x) \, d\mu \\
& = \liminf_{k \to +\infty}\lim_{n \to + \infty}\left ( \int_{\Omega}
\varphi_k(x) f(\chi_n(x),\E u_n (x)) \, dx + 
\int_{\Omega}\varphi_k(x) \, d |D\chi_n| \right )\\
& \leq \lim_{n \to + \infty} \left (\int_{\Omega}
f(\chi_n(x),\E u_n (x)) \, dx + |D\chi_n|(\Omega) \right )
=  {\cal F}(\chi,u;\Omega),
\end{align*}
so that
\begin{equation}\label{mu1}
 \mu(\overline{\Omega})\leq {\cal F}(\chi, u; \Omega).
\end{equation}
On the other hand, by the upper semicontinuity of weak $\ast$ convergence of 
measures on compact sets, 
for every open set $V \subset \Omega$, it follows that
\begin{equation}\label{mu2}
{\cal F}(\chi, u; V) \leq \liminf_{n \to +\infty} F(\chi_n, u_n; V) 
= \liminf_{n \to +\infty} \mu_n(V) 
\leq \limsup_{n \to +\infty} \mu_n(\overline{V})
\leq \mu(\overline{V}).
\end{equation}
Now let  $V \in \mathcal O(\O)$ and $\e>0$ be fixed and consider an open set  
$S \subset \subset V$ such that $\mu(V \setminus S) < \e$. Then
\begin{equation}\label{mu3}
\mu(V) \leq \mu(S) + \e = \mu(\overline \O) - \mu(\overline \O \setminus S) + \e,
\end{equation}
and so, by \eqref{mu3}, \eqref{mu1}, \eqref{mu2} and Proposition \ref{nestedsa} 
we have
\begin{equation*}
\mu(V)\leq \mu({\overline \Omega})-  \mu( {\overline\Omega} \setminus S) 
+ \varepsilon 
\leq {\cal F}(\chi, u; \Omega)- 
{\cal F}(\chi, u; \Omega \setminus {\overline S}) + \varepsilon 
\leq {\cal F}(\chi, u; V)+\varepsilon.
\end{equation*}
Letting $\e \to 0^+$, we obtain 
$$
\mu(V)\leq {\cal F}(\chi, u; V),
$$
whenever $V$ is an open set such that $V \subset \subset \Omega$.
For a general open subset $V \subset \Omega$ we have
$$
\mu(V) = \sup \{\mu(O) : O \subset \subset V \}
\leq \sup \{{\cal F}(\chi, u; O) : O \subset \subset V \}
\leq {\cal F}(\chi, u; V).
$$ 

It remains to show that ${\cal F}(\chi, u; U) \leq \mu(U)$, 
$\forall \, U \in \mathcal O(\O)$. Fix $\e > 0$ and choose $V, S \in \mathcal O(\O)$
such that $S \subset \subset V \subset \subset U$ and 
$\mathcal L^N(U \setminus \overline S) + |Eu|(U \setminus \overline S) 
+ |D\chi|(U \setminus \overline S) < \e$. By Proposition \ref{firstprop} $i)$,
\eqref{mu2} and the nested subadditivity result, it follows that
\begin{align*}
{\cal F}(\chi, u; U) &\leq {\cal F}(\chi, u; V) + 
{\cal F}(\chi, u; U \setminus {\overline S}) \\
&\leq \mu(\overline V) +
C\left(\mathcal L^N(U \setminus \overline S) + |Eu|(U \setminus \overline S) 
+ |D\chi|(U \setminus \overline S)\right) \leq \mu(U) + C\e,
\end{align*}
so it suffices to let $\e \to 0^+$ to conclude the proof.
\end{proof}

Combining the arguments given in the proofs of Propositions \ref{slicing} and
\ref{nestedsa} it is possible to obtain the following refined version of Proposition
\ref{slicing}.

\begin{proposition}\label{newslicing}
Let $A \in \mathcal O (\Omega)$ and assume that $W_0, W_1$ satisfy the growth condition \eqref{growth}. Let
$\{u_n\}, \{v_n\} \subset BD(A;\mathbb R^N)$ and $\{\chi_n\}, \{\theta_n\} \subset BV(A;\{0,1\})$ be sequences
satisfying $u_{n} - v_n \to 0$ in $L^{1}(A;\mathbb{R}^{N})$,
$\chi_{n} - \theta_n \to 0$ in $L^1(A;\{0,1\})$, 
$\sup_n |E^su_n|(A) < + \infty$, $|Ev_n| \overset{\ast}{\rightharpoonup} \mu$,
$|Ev_n| \to \mu(A)$, $\sup_n |D\chi_n|(A) < + \infty$ and 
$\sup_n |D\theta_n|(A) < + \infty$. Then there exist subsequences $\{v_{n_k}\}$ of $\{v_n\}$, 
$\{\theta_{n_k}\}$ of $\{\theta_n\}$ and there exist sequences 
$\{w_k\} \subset BD(A)$, $\{\eta_k\} \subset BV(A;\{0,1\})$
such that $w_k = v_{n_k}$ near $\partial A$, 
$\eta_k = \theta_{n_k}$ near $\partial A$,
$w_k - v_{n_k} \to 0$ in $L^{1}(A;\mathbb{R}^{N})$,
$\eta_{k} - \theta_{n_k} \to 0$ in $L^1(A;\{0,1\})$ and
$$\limsup_{k\to +\infty}F(\eta_{k},w_k;A) \leq 
\liminf_{n\to +\infty}F(\chi_n,u_n;A).$$
\end{proposition}

As in Proposition \ref{slicing}, the new sequence $\{w_k\}$ has the same regularity as the original sequences $\{u_n\}, \{v_n\}$ as it is obtained through a convex combination of these ones using smooth cut-off functions.

\smallskip

The following proposition, whose proof is standard (cf. for instance \cite[Lemma 3.1]{RZ} or \cite[Proposition 2.14]{CZasy}),  allows us to assume without loss of generality that $f$ is symmetric quasiconvex. 

\begin{proposition}\label{Frel=Frel**}
Let $W_0$ and $W_1$ be continuous functions satisfying \eqref{growth} and consider the functional 
$F:BV(\Omega;\{0,1\})\times BD(\Omega)\times \mathcal O(\Omega)$ 
defined in \eqref{F}.
Consider furthermore the relaxed functionals given in \eqref{calF} and
\begin{align}\label{Fsqxrelax}
\mathcal F_{SQf}(\chi,u;A) & :=\inf\Big\{\liminf_{n\to +\infty}
\int_A  SQf(\chi_n(x),\E u_n(x)) \, dx + |D \chi_n|(A) : \\
& (\chi_n,u_n) \in BV(A;\{0,1\})\times LD(A),
u_{n}\to u\text{ in }
L^{1}(A;\mathbb{R}^{N}), \chi_{n} \to \chi \text{ in } L^1(A;\{0,1\})\Big\}.\nonumber
\end{align}

Then, ${\mathcal F}(\cdot,\cdot;\cdot)$ coincides with 
${\mathcal F}_{SQf}(\cdot, \cdot; \cdot)$ in 
$BV(\Omega;\{0,1\})\times BD(\Omega)\times \mathcal O(\Omega)$.
\end{proposition}

In the sequel we rely on the result of Proposition \ref{Frel=Frel**} and assume
that $f$ is symmetric quasiconvex. Together with \eqref{growth}, this entails the Lipschitz continuity of $f$ with respect to the second variable (see \cite{DPR2}). Under this quasiconvexity hypothesis, assuming in addition that \eqref{finfty} holds and taking also into account Proposition \ref{propperH5}, we
recall (cf. \eqref{KSQ}) that our relaxed surface energy density is given by
\begin{equation}\label{K}
K(a,b,c,d,\nu):=\inf\left\{  
	\displaystyle\int_{Q_{\nu}}
	f^{\infty}(\chi(x),\E u(x)) \, dx+|D\chi|(Q_{\nu}):\left(\chi,u\right)
\in\mathcal{A}(a,b,c,d,\nu)\right\}, 
\end{equation}
where, for $(a,b,c,d,\nu)  \in \{0,1\} \times \{0,1\}  \times \mathbb{R}^{N} \times\mathbb{R}^{N} \times S^{N-1},$ the set of admissible functions is
\begin{align}\label{adm}
\mathcal{A}(a,b,c,d,\nu)   &  :=\bigg\{  \left(\chi,u\right)
\in BV_{\rm loc}\left(S_{\nu};\{0,1\}\right)  \times 
W^{1,1}_{\rm loc}\left(S_{\nu};\mathbb{R}^{N}\right)  : \\
&   (\chi(y),u(y)) = (a,c) \text{ if } y\cdot\nu=\frac{1}{2},
~(\chi(y),u(y)) = (b,d) \text{ if } y\cdot\nu=-\frac{1}{2},\nonumber\\
&   (\chi, u)\text{ are 1-periodic in the directions of }\nu_{1},\dots,\nu_{N-1}\bigg\}  ,\nonumber
\end{align}
$\left\{\nu_{1},\nu_{2},\dots,\nu_{N-1},\nu\right\}$ is an orthonormal
basis of $\mathbb{R}^{N}$ and $S_\nu$ is the strip given by
$$S_\nu = \left\{x \in \Rb^N : |x \cdot \nu| < \frac{1}{2}\right\}.$$

The following result provides an alternative characterisation of $K(a,b,c,d,\nu)$
which will be useful to obtain the surface term of the relaxed energy, under hypothesis \eqref{finfty}.
To this end, given $(a,b,c,d,\nu)  \in \{0,1\} \times \{0,1\}  \times \mathbb{R}^{N} \times\mathbb{R}^{N} \times S^{N-1},$ we consider the functions 
\begin{equation}\label{targets}
\chi_{a,b,\nu}(y) := \begin{cases}
a, & {\rm if } \; y \cdot \nu > 0\\
b, & {\rm if } \; y \cdot \nu < 0
\end{cases}
\; \; \; {\rm and } \; \; \;
u_{c,d,\nu}(y) := \begin{cases}
c, & {\rm if } \; y \cdot \nu > 0\\
d, & {\rm if } \; y \cdot \nu < 0.
\end{cases}
\end{equation}

\begin{proposition}\label{Ktilde}
For every $(a,b,c,d,\nu)  \in \{0,1\} \times \{0,1\}  \times \mathbb{R}^{N} \times\mathbb{R}^{N} \times S^{N-1}$ we have
$$K(a,b,c,d,\nu) = \widetilde{K}(a,b,c,d,\nu)$$
where
\begin{align}\label{tilK}
\widetilde{K}(a,b,c,d,\nu) & := \inf\bigg\{\liminf_{n \to + \infty}
\left[\displaystyle\int_{Q_{\nu}}f^{\infty}(\chi_n(x),\E u_n(x)) \, dx
+ |D\chi_n|(Q_{\nu})\right] : \chi_n \in BV\left(Q_{\nu};\{0,1\}\right), \nonumber\\
& \hspace{1,3cm} u_n \in W^{1,1}\left(Q_{\nu};\mathbb{R}^{N}\right), \chi_n \to \chi_{a,b,\nu}
\; {\rm in } \; L^1(Q_{\nu};\{0,1\}), u_n \to u_{c,d,\nu} \; {\rm in } \; L^1(Q_{\nu};\Rb^N) \bigg\}. 
\end{align}
\end{proposition}
\begin{proof}
The conclusion follows as in \cite[Proposition 3.5]{BBBF}, by proving a double inequality. 

To show that $K(a,b,c,d,\nu) \leq \widetilde{K}(a,b,c,d,\nu)$ we take sequences $\{\chi_n\}$, $\{u_n\}$ as in the definition of $\widetilde{K}(a,b,c,d,\nu)$ and use
Proposition \ref{newslicing}, applied to $\{\chi_n\}$, $\{\chi_{a,b,\nu}\}$,
$\{u_n\}$ and $\{v_n\}$, where $v_n$ is a regularization of $u_{c,d,\nu}$ which preserves its boundary values (cf. Theorem \ref{densitysmooth}).

The reverse inequality is based on the periodicity of the admissible functions for $K(a,b,c,d,\nu)$, together with the Riemann-Lebesgue Lemma.
\end{proof}

\section{Proof of the Main Theorem}\label{mainthm}

Given $\chi \in BV(\Omega;\{0,1\})$ and $u \in BD(\Omega)$, by Proposition \ref{traceRm} we know that $\mathcal F(\chi,u,;\cdot)$ is the restriction to $\mathcal O(\Omega)$ of a Radon measure $\mu$. By Proposition~\ref{firstprop} $i)$ we may decompose $\mu$ as
$$\mu = \mu^a \mathcal L ^N + \mu^j 
+ \mu^c, \;\; \hbox{ with } \mu^j \ll |E^j u|+ |D \chi|.$$
Our aim in this section is to characterise the density $\mu^a$ and the measures
$\mu^j$ and $\mu^c$.

We point out that the measure $\mu^j$ is given by 
$\sigma^j \mathcal H^{N-1} \lfloor (J_\chi \cup J_u)$, for a certain density $\sigma^j$. Indeed, due to the fact that, for $BV$ functions,  
$\mathcal H^{N-1}(S_u \setminus J_u) = 0$, the measure $|D\chi|$ is concentrated on 
$J_\chi$ apart from an $\mathcal H^{N-1}$-negligible set, whereas, by 
\cite[Remark 4.2 and Proposition 4.4]{ACDM},
$|E^ju|$ is concentrated on $J_u$ and it is the only part of the measure $Eu$ that is concentrated on $(n-1)$-dimensional sets.

\subsection{The Bulk Term}\label{bulk}

\begin{proposition}\label{lbbulk}
Let $u \in BD(\Omega)$, $\chi \in BV(\Omega;\{0,1\})$ and let $W_0$ and $W_1$ be continuous functions satisfying \eqref{growth}. 
Assume that $f$ given by \eqref{density} is symmetric quasiconvex. 
Then, for $\LL ^N$ a.e. $x_0 \in \O$,
$$ \mu^a(x_0) = \dfrac{d\mathcal F(\chi, u; \cdot)}{d {\mathcal L}^N}(x_0) \geq f(\chi(x_0), \mathcal E u(x_0)).$$
\end{proposition}
\begin{proof}
Let  $x_0 \in \O$ be a point satisfying
\begin{equation}\label{RNdev}
\mu^a(x_0) = \frac{d \mu}{d \mathcal L^N}(x_0)=
\lim_{\e \to 0^+}\frac{\mu(Q(x_0,\e))}{\e^ N} 
\; \; \mbox{exists and is finite}
\end{equation}
and
\begin{equation}\label{440bis}
\frac{d |E^su|}{d\mathcal{L}^N}(x_0)= 0, \; \; 
\frac{d |D\chi|}{d\mathcal{L}^N}(x_0)=0.
\end{equation}

Furthermore, we choose $x_0$ to be a point of approximate continuity for $u$, for $\E u$ and for $\chi$, namely we assume that 
\begin{equation}\label{uac}
	\lim_{\e\rightarrow 0^+}
	\frac{1}{\e^{N}}\int_{Q\left(x_0,\e\right)}
	\left\vert u(x)-u(x_0)\right\vert \,dx = 0,
\end{equation}
\begin{equation}\label{364}
\lim_{\e \to 0^+} \frac{1}{\e^N} \int_{Q(x_0, \e)} 
|\mathcal E u(x) - \mathcal Eu(x_0)|\, dx = 0
\end{equation}
and
\begin{equation}\label{chiad}
\lim_{\e\rightarrow 0^+}
\frac{1}{\e^{N}}\int_{Q\left(x_0,\e\right)}
\left\vert \chi(x)-\chi(x_0)\right\vert \,dx = 0.
\end{equation}

We observe that the above properties hold for ${\mathcal L}^N$ a.e. $x_0 \in \Omega$ (applying, for instance, \cite[eq. (2.5)]{ACDM} to $u$, 
$\mathcal E u$ and $\chi$).

Assuming that the sequence $\e_k \to 0^+$ is chosen in such a way that 
$\mu(\partial Q(x_0,\e_k)) = 0$, we have
\begin{equation*}
\begin{split}
\mu^a(x_0) = \lim_{\e_k \to 0^+}\frac{\mu (Q(x_0,\e_k))}{\e_k^N} &=  \lim_{\e_k,n}\left[\frac{1}{\e_k^N}\int_{Q(x_0,\e_k)}
f(\chi_n(x),\mathcal E u_n(x))\,dx  +  |D \chi_n|(Q(x_0,\e_k))\right]\\
&\geq \lim_{\e_k,n}\int_Q f(\chi_n(x_0+\e_k y), \mathcal E u_n(x_0+\e_k y))\,dy,
\end{split}
\end{equation*}
where $\chi_n \in BV(Q(x_0,\e_k);\{0,1\})$, $\chi_n \to \chi$ in $L^1(Q(x_0,\e_k);\{0,1\})$ and $u_n \in W^{1,1}(Q(x_0,\e_k);\Rb^N)$, $u_n \to u$ in $L^1(Q(x_0,\e_k);\Rb^N)$. 

Defining 
$$\chi_{n,\e_k}(y):=\chi_n(x_0+\e_k y)- \chi(x_0),$$
it follows by \eqref{chiad} that
\begin{equation}\label{chindeltalim}
\begin{split}
\lim_{\e_k,n}\|\chi_{n,\e_k}\|_{L^1(Q)} = &\lim_{\e_k,n}
\int_Q|\chi_n(x_0+\e_k y)-\chi(x_0)|\,dy \\
=&\lim_{\e_k,n}\frac{1}{\e_k^{N}}\int_{Q(x_0,\e_k)}|\chi_n(x)-\chi(x_0)|\,dx\\
=&\lim_{\e_k \to 0^+}\frac{1}{\e_k^{N}}\int_{Q(x_0,\e_k)}|\chi(x)-\chi(x_0)|\, dx =0. 
\end{split}
\end{equation}

Analogously, letting
$$u_{n,\e_k}(y):=\frac{u_n(x_0+\e_k y)- u(x_0)}{\e_k},$$
then
$\mathcal E u_{n,\e_k}(y)= \mathcal E u_n(x_0+\e_k y)$
and, since $u_{n,\e_k}\in W^{1,1}(\Omega;\Rb^N)$, 
$E u_{n,\e_k}= \mathcal E u_{n,\e_k} \LL^N$. 

Moreover, arguing as in the proof of \cite[Proposition 4.1]{BFT}, exploiting the coercivity of $f$ in the second variable and Theorems \ref{Thm2.8BFT} and \ref{THM2.10BFT}, we conclude that there exists a function $v \in BD(\Omega)$, such that 
$$
\lim_{\e_k,n}\|u_{n,\e_k}- P(u_{n,\e_k})-v\|_{L^1(Q;\mathbb R^N)} = 0,
$$
where $P$ is the projection of $BD(\O)$ onto the kernel of the operator $E$.
Furthermore, given that the point $x_0$ was chosen to satisfy \eqref{440bis} and \eqref{364}, it was shown in \cite[Proposition 4.1, (4.8)]{BFT} that
\begin{equation}\label{measw*}
Ev = \mathcal E u(x_0)\mathcal L^N.
\end{equation}

Therefore, a diagonalisation argument allows to extract subsequences 
$u_k := u_{n_k,\e_k} - P(u_{n_k,\e_k})$ and
$\chi_k := \chi_{n_k,\e_k}$, such that 
\begin{equation}\label{radii}
\begin{split}
& \lim_{k \to +\infty}\|\chi_{k}\|_{L^1(Q)} = 0,\\
& \lim_{k \to +\infty} \|u_{k}- v\|_{L^1(Q;\mathbb R^N)} = 0
\end{split}
\end{equation}
and 
\begin{align}\label{8}
\mu^a(x_0) = \frac{d\mu}{d\mathcal{L}^N}(x_0)\geq \lim_{k \to + \infty}
\int_Q f(\chi(x_0)+ \chi_{k}(y),\mathcal E u_{k}(y)) \, dy.
\end{align}

Our next step is to fix $\chi(x_0)$ in the first argument of $f$ in the previous integral. To this end we make use of Chacon's Biting Lemma 
(see \cite[Lemma 5.32]{AFP}). Indeed, by the coercivity hypothesis \eqref{growth} and
\eqref{8}, the sequence $\{\E u_k\}$ is bounded in $L^1(Q;\Rb^{N\times N}_s)$ so the Biting Lemma guarantees the existence of a (not relabelled) subsequence of $\{u_k\}$ and of a decreasing sequence of Borel sets $D_r$, such that 
$\displaystyle \lim_{r \to + \infty}{\mathcal L}^N(D_r) = 0$ and the sequence 
$\{\mathcal Eu_k\}$ is equiintegrable in $Q\setminus D_r$, 
for any $r \in \mathbb N$.

Since $f \geq 0$, by \eqref{G} and \eqref{8}, we have
\begin{align}\label{9} 
\mu^a(x_0) & \geq \lim_{k \to + \infty}\int_{Q\setminus D_r}
f(\chi(x_0) + \chi_k(y), \mathcal E u_{k}(y))\, dy \nonumber\\
& \geq \lim_{k \to + \infty}
\left\{\int_{Q\setminus D_r} f(\chi(x_0), \mathcal E u_{k}(y)) \, dy - \int_{Q\setminus D_r} C |\chi_{k}(y)|\cdot (1 + |\mathcal E u_k(y)|) \, dy \right\}\nonumber \\
& \geq \lim_{k \to + \infty} \int_{Q \setminus D_r} 
f(\chi(x_0), \mathcal E u_{k}(y)) \, dy - \limsup_{k \to + \infty} 
\int_{Q\setminus D_r} C  |\chi_{k}(y)| \cdot |\mathcal E u_k(y)| \, dy, 
\end{align}
where we used \eqref{chindeltalim}.

We claim that for each $j \in \mathbb N$, there exist $k=k(j)$ and 
$r_j \in \mathbb N$, such that
\begin{equation} \label{489} 
\int_{Q\setminus D_{r_j}} f(\chi(x_0), \E u_{k(j)}(y)) \, dy \geq 
\int_Q f(\chi(x_0), \E u_{k(j)}(y)) \, dy - \frac{C}{j}.
\end{equation}
In light of \eqref{growth}, in order to guarantee that \eqref{489} holds, it suffices to show that for each $j \in \Nb$, there exist $k=k(j)$ and $r_j \in \Nb$, such that
\begin{equation}\label{77}
\int_{D_{r_j}} 1 + |\mathcal E u_{k(j)}(y)| \, dy \leq \frac{1}{j}.
\end{equation}
Suppose not. Then, there exists $j_0 \in \Nb$ such that, for all $r, k \in \Nb$, 
\begin{equation} \label{497}  
\int_{D_r}1 + |\mathcal E u_{k}(y)| \, dy  > \frac{1}{j_0}
\end{equation}
which contradicts the equiintegrability of the constant sequence $\{1 + |\E u_k|\}$,
for $k$ fixed, and the fact that 
$\displaystyle \lim_{r \to + \infty}{\mathcal L}^N(D_r) = 0$.

For this choice of $k(j)$ and $r_j$, we now estimate the last term in \eqref{9}. Since
$|\chi_{k(j)}| \to 0$, as $j \to + \infty$, in $L^1(Q)$, this sequence also converges to zero in measure. Thus, denoting by 
$$\displaystyle A_{k(j)} : = \left\{x \in Q \setminus D_{r_j} : |\chi_{k(j)}(x)| = 1 \right\},$$ 
it follows that for every $\delta > 0$, there
exists $j_0 \in \Nb$ such that $\LL^N(A_{k(j)}) < \delta$, for all $j > j_0$.

On the other hand, because the sequence 
$\{\mathcal Eu_{k(j)}\}$ is equiintegrable in $Q\setminus D_{r_j}$, we know that 
for every $\e > 0$, there exists $\delta = \delta(\e) > 0$ such that for any measurable set $A \subset Q\setminus D_{r_j}$ with $\LL^N(A) < \delta(\e)$ we have
$\displaystyle \int_{A}|\mathcal Eu_{k(j)}(y)| \, dy < \e$.
Choosing $j$ large enough so that $\LL^N(A_{k(j)}) < \delta(\e)$ we obtain
$\displaystyle \int_{A_{k(j)}}|\mathcal Eu_{k(j)}(y)| \, dy < \e$ and hence
\begin{equation}\label{11}
\int_{Q\setminus D_{r_j}}
|\chi_{k(j)}(y) \cdot |\mathcal Eu_{k(j)}(y)| \, dy < \e,
\end{equation}
for every sufficiently large $j$.

Therefore, up to the extraction of a further subsequence, and denoting in what follows $\chi_j := \chi_{k(j)}$, $v_j : = u_{k(j)}$ and 
$ D_j:=D_{r_j}$,  \eqref{9}, \eqref{489} and \eqref{11} yield
\begin{align*}
\mu^a(x_0) = \frac{d\mu}{d\mathcal{L}^N}(x_0) & \geq \liminf_{j \to + \infty}
\left(\int_Q f(\chi(x_0),\mathcal E v_j(y)) \, dy - \frac{C}{j}\right) 
-\limsup_{j \to + \infty}\int_{Q\setminus D_{j}} 
C |\chi_j(y)| \cdot |\mathcal E v_j(y)|\, dy \\
& \geq \liminf_{j \to + \infty} \int_Q f(\chi(x_0),\mathcal E v_j(y)) \, dy - \e.
\end{align*}

Since $v_j \to v$ in $L^1(Q;\Rb^N)$, Proposition \ref{slicing} allows us to assume, without loss of generality, that $v_j = v$ on $\partial Q$. Hence, using the symmetric quasiconvexity of $f$ in the second variable, which also holds for test functions in $LD_{\rm per}(Q)$ (cf. Remark \ref{LDper}), and \eqref{measw*}, we obtain
\begin{align*}
\mu ^a(x_0) &\geq 
\liminf_{j \to + \infty} \int_Q f(\chi(x_0),\mathcal E v_j(y)) \, dy - \e \\
&\geq \liminf_{j \to + \infty} 
\int_Q f(\chi(x_0),\E u(x_0) + \mathcal E (v_j - v)(y)) \, dy - \e \\
& \geq f(\chi(x_0),\mathcal E u(x_0)) - \e,
\end{align*}
so to conclude it suffices to let $\e \to 0^+$.
\end{proof}

\begin{proposition}\label{upbbulk}
Let $u \in BD(\Omega)$, $\chi \in BV(\Omega;\{0,1\})$ and let $W_0$ and $W_1$ be continuous functions satisfying \eqref{growth}. 
Let $f$ be given by \eqref{density} and assume that $f$ is symmetric quasiconvex. 
Then, for $\LL^N$ a.e. $x_0 \in \O$,
$$ \mu^a(x_0) = \dfrac{d\mathcal F(\chi, u; \cdot)}{d {\mathcal L}^N}(x_0) \leq f(\chi(x_0), \mathcal E u(x_0)).$$
\end{proposition}
\begin{proof}
Choose a point $x_0 \in \O$ such that \eqref{uac},\eqref{364}, \eqref{chiad} hold, 
\begin{equation}\label{x02}
\lim_{\varepsilon \to 0^+}
\frac{1}{\e ^N}|E^s u|(Q(x_0, \e)) = 0,
\end{equation}
\begin{equation}\label{x03}
\lim_{\varepsilon \to 0^+}
\frac{1}{\e ^N}|D \chi|(Q(x_0, \e)) = 0,
\end{equation}
and, furthermore, such that
\begin{equation}\label{x05}
\mu^a(x_0) = \lim_{\varepsilon \to 0^+}
\frac{\mathcal F(\chi, u;Q(x_0,\e))}{\e ^N} \; \; \; \mbox{exists and is finite},
\end{equation}
where the sequence of $\e \to 0^+$ is chosen so that $|E u|(\partial Q(x_0,\e)) = 0$.
Notice that $\LL^N$ almost every point $x_0 \in \O$ satisfies the above properties.

For the purposes of this proof we assume that $\chi(x_0) = 1$, the case 
$\chi(x_0) = 0$ is treated in a similar fashion. Thus, it follows from \eqref{chiad} that
\begin{equation}\label{x07}
\lim_{\varepsilon \to 0^+}
\frac{1}{\e ^N}\LL^N\big(Q(x_0, \e) \cap \{\chi = 0\}\big) = 0.
\end{equation}
Using the symmetric quasiconvexity of $f$, fix $\delta > 0$ and let 
$\phi \in C^{\infty}_{\rm per}(Q;\Rb^N)$ be such that
\begin{equation}\label{phi}
\int_Qf(\chi(x_0),\E u(x_0) + \E \phi(x)) \, dx \leq f(\chi(x_0),\E u(x_0)) + \delta.
\end{equation}
We extend $\phi$ to $\Rb^N$ by periodicity, define 
$\phi_n(x) := \dfrac{1}{n}\phi(nx)$ and consider the sequence of functions in $W^{1,1}(Q(x_0,\e);\Rb^N)$ given by
$$u_{n,\e}(x) : = (\rho_n * u)(x) + \e \phi_n\left(\frac{x - x_0}{\e}\right).$$
The periodicity of $\phi$ ensures that, as $ n \to + \infty$, $u_{n,\e} \to u$ in 
$L^1(Q(x_0,\e);\Rb^N)$ and so, letting $\chi_n = \chi$, $\forall n \in \Nb$, the
sequences $\{u_{n,\e}\}_n$ and $\{\chi_n\}_n$ are admissible for 
$\mathcal F(\chi,u;Q(x_0,\e))$. Hence, by \eqref{x03}, we have
\begin{align*}
\mu^a(x_0) &= \lim_{\varepsilon \to 0^+}
\frac{\mathcal F(\chi, u;Q(x_0,\e))}{\e ^N} \leq 
\liminf_{\varepsilon \to 0^+}\liminf_{n \to +\infty}\frac{1}{\e^N}
\Big(\int_{Q(x_0, \e)}f(\chi(x),\E u_{n,\e}(x)) \, dx + |D\chi|(Q(x_0,\e))\Big)\\
&=\liminf_{\varepsilon \to 0^+}\liminf_{n \to +\infty}\frac{1}{\e^N}
\int_{Q(x_0, \e)}f\left(\chi(x),\E u_{n,\e}(x)\right) \, dx \\
&\leq \limsup_{\varepsilon \to 0^+}\limsup_{n \to +\infty}\frac{1}{\e^N}
\int_{Q(x_0, \e)}f\left(\chi(x_0),\E u(x_0) + \E \phi_n\left(\frac{x-x_0}{\e}\right)\right) \, dx \\
& + \limsup_{\varepsilon \to 0^+}\limsup_{n \to +\infty}\frac{1}{\e^N}
\int_{Q(x_0, \e)} f\left(\chi(x),\E u_{n,\e}(x)\right) - 
f\left(\chi(x_0),\E u_{n,\e}(x))\right)\, dx \\
& + \limsup_{\varepsilon \to 0^+}\limsup_{n \to +\infty}\frac{1}{\e^N}
\int_{Q(x_0, \e)} f\left(\chi(x_0),\E u_{n,\e}(x)\right) - 
f\left(\chi(x_0),\E u(x_0) + \E \phi_n\left(\frac{x-x_0}{\e}\right)\right)\, dx \\
& =: I_1 + I_2 + I_3.
\end{align*}
By changing variables, using the periodicity of $\phi$ and \eqref{phi}, it follows that
\begin{align*}
I_1 &= \limsup_{n \to +\infty}\int_Q f(\chi(x_0),\E u(x_0) + \E \phi_n(y)) \, dy
= \limsup_{n \to +\infty}\int_Q f(\chi(x_0),\E u(x_0) + \E \phi(ny)) \, dy\\
&= \limsup_{n \to +\infty}\int_Q f(\chi(x_0),\E u(x_0) + \E \phi(x)) \, dx 
\leq f(\chi(x_0),\E u(x_0)) + \delta.
\end{align*}
Consequently, to complete the proof it remains to show that $I_2 = I_3 = 0$ and finally to let $\delta \to 0^+$. 
To conclude that $I_3 = 0$ we reason exactly as in \cite[Proposition 4.2]{BFT} since
$\chi(x_0)$ is fixed in both terms of the integrand. As for $I_2$, since 
$\chi(x_0) = 1$, we have by \eqref{G},
\begin{align*}
I_2 &=  \limsup_{\varepsilon \to 0^+}\limsup_{n \to +\infty}\frac{1}{\e^N}
\int_{Q(x_0, \e) \cap \{\chi = 0\}} f\left(0,\E u_{n,\e}(x)\right) - 
f\left(1,\E u_{n,\e}(x))\right)\, dx \\
&\leq \limsup_{\varepsilon \to 0^+}\limsup_{n \to +\infty}\frac{C}{\e^N}
\int_{Q(x_0, \e) \cap \{\chi = 0\}} 1 + 
\left |\E (u*\rho_n)(x) + \E \phi_n\left(\frac{x-x_0}{\e}\right)\right| \, dx,
\end{align*}
where, by periodicity and the Riemann-Lebesgue Lemma,
\begin{align*}
\limsup_{\varepsilon \to 0^+}\limsup_{n \to +\infty}\frac{C}{\e^N}
\int_{Q(x_0, \e) \cap \{\chi = 0\}}
&\left |\E \phi_n\left(\frac{x-x_0}{\e}\right)\right| \, dx 
= \limsup_{\varepsilon \to 0^+}\limsup_{n \to +\infty}
C \int_{Q \cap \{y : \chi(x_0 + \e y) = 0\}}\left |\E\phi(ny)\right| \, dy\\
&\hspace{2cm} = \limsup_{\varepsilon \to 0^+}C \int_{Q \cap \{y : \chi(x_0 + \e y) = 0\}}
\left(\int_Q|\E \phi(x)| \, dx\right) \, dy\\
&\hspace{2cm} = \limsup_{\varepsilon \to 0^+}\frac{C}{\e^N}
\LL^N\left(Q(x_0,\e) \cap \{\chi = 0\}\right)\int_Q|\E \phi(x)| \, dx = 0
\end{align*}
by \eqref{x07}. 
On the other hand, since $|E u|$ does not charge the boundary of $Q(x_0,\e)$, using Lemma \ref{lemma2.2BFT}, \eqref{x02}, \eqref{364} and \eqref{x07},
it follows that
\begin{align*}
& \limsup_{\varepsilon \to 0^+}\limsup_{n \to +\infty}\frac{C}{\e^N}
\int_{Q(x_0, \e) \cap \{\chi = 0\}}
\left |\E (u*\rho_n)(x)\right| \, dx \\ 
& \leq \limsup_{\varepsilon \to 0^+}\limsup_{n \to +\infty}\frac{C}{\e^N}
\int_{Q(x_0, \e + \frac{1}{n}) \cap \{\chi = 0\}}d|E u|(x)  \\
&= \limsup_{\varepsilon \to 0^+}\frac{C}{\e^N}
\int_{Q(x_0, \e) \cap \{\chi = 0\}}|\E u(x)| \, dx  \\
& \leq \limsup_{\varepsilon \to 0^+}\frac{C}{\e^N}
\int_{Q(x_0, \e)}|\E u(x) - \E u(x_0)| \, dx +
\limsup_{\varepsilon \to 0^+}\frac{C |\E u(x_0)|}{\e^N}
\LL^N\left(Q(x_0, \e) \cap \{\chi = 0\}\right) = 0.
\end{align*}
Therefore, a final application of \eqref{x07} allows us to conclude that $I_2 = 0$.
\end{proof}

\begin{remark}\label{norestsqcx}
{\rm We stress that the symmetric quasiconvexity hypothesis on $f$ in Proposition \ref{upbbulk} is not a restriction for the proof of Theorem \ref{main}, in view of Proposition \ref{Frel=Frel**}.}
\end{remark}

\subsection{The Cantor Term}\label{Cantor}

This section is devoted to the identification of the density of $\mathcal F$ in \eqref{calFint} with respect to $|E^c u|$. To this end, we start by observing that, by virtue of Proposition \ref{Frel=Frel**}, there is no loss of generality in assuming that $f$ is symmetric quasiconvex. If this symmetric quasiconvexity hypothesis on $f$ is omitted, the result of the next proposition holds provided we replace $f^\infty$ by $(SQf)^\infty$, whereas, due to the inequality 
$(SQf)^{\infty} \leq f^{\infty}$,  \eqref{ubCantor} holds as stated.

\begin{proposition}\label{lbCantor}
Let $u \in BD(\Omega)$, $\chi \in BV(\Omega;\{0,1\})$ and let $W_0$ and $W_1$ be continuous functions satisfying \eqref{growth}. 
Assume that $f$ given by \eqref{density} is symmetric quasiconvex.
Then, for $|E^cu|$ a.e. $x_0 \in \O$,
$$ \mu^c(x_0) = \dfrac{d\mathcal F(\chi, u; \cdot)}{d |E^cu|}(x_0) \geq f^{\infty}\left(\chi(x_0), \frac{d E^c u}{d |E^c u|}(x_0)\right).$$
\end{proposition}
\begin{proof}
Let  $x_0 \in \O$ be a point satisfying \eqref{uac}, \eqref{chiad} and
\begin{equation}\label{Cdens}
\mu^c(x_0) = \dfrac{d\mathcal F(\chi, u; \cdot)}{d |E^cu|}(x_0)
= \frac{d \mu}{d |E^cu|}(x_0) = 
\lim_{\e \to 0^+}\frac{\mu(Q(x_0,\e))}{|E^cu|(Q(x_0,\e))} 
\; \; \mbox{exists and is finite,}
\end{equation}
these properties hold for $|E^c u|$ a.e. $x_0 \in \Omega$. Indeed, by \cite[Theorem 6.1]{ACDM}, 
$|E u|(S_u\setminus J_u)=0$, thus $|E^c u|(S_u\setminus J_u)=0$. Hence, by \cite[Propositions 3.5 and 4.4]{ACDM}, we have 
$$|E^c u|(S_u)= |E^c u|(J_u)+ |E^c u|(S_u\setminus J_u)=0,$$ which justifies the validity of \eqref{uac}. As for \eqref{chiad}, this is a well known property of $BV$ functions (cf. \cite{AFP}).

We define
$$f_0(\xi) = f(0, \xi) \mbox{ and } f_1(\xi) = f(1, \xi), \forall \xi \in \Rb^{N \times N}_s$$
and we consider the auxiliary functionals
\begin{align}\label{auxf}
\mathcal F_i(u; A) &:= \inf \Big \{\liminf_{n \to + \infty} 
\int_A f_i (\E u_n(x)) \, dx :  u_n \in W^{1, 1}(A; \Rb^N), u_n \to u \; {\rm in} \; L^1(A;\Rb^N)\Big \}, \, i=0,1. 
\end{align}
Referring to Theorem 6.1, Remark 6.4 and Corollary 6.8 in \cite{CFVG},
$\mathcal F_i(u;\cdot)$, $i=0,1$,  
are the restriction to $\mathcal O(\Omega)$ of Radon measures whose densities with respect to $|E^c u|$ are given by 
\begin{equation}\label{CdensFi}
\frac{d \mathcal F_i(u;\cdot)}{d|E^c u|}(x_0) 
= f_{i}^\infty \Big(\frac{d E^c u}{d|E^c u|}(x_0)\Big) =
f^\infty \Big(i,\frac{d E^c u}{d|E^c u|}(x_0)\Big)
\end{equation}
for $|E^c u|$ a.e. $x_0 \in \Omega$. Choose $x_0$ so that it also satisfies \eqref{CdensFi}, $i=0,1$.

In what follows we assume, without loss of generality, that $\chi(x_0)=1$, the case $\chi(x_0)=0$ can be treated similarly. Bearing this choice in mind we work with the 
functional \eqref{auxf} and we will make use of \eqref{CdensFi}, with $i=1$.
\color{black}
Selecting the sequence $\e_k \to 0^+$ in such a way that 
$\mu(\partial Q(x_0,\e_k)) = 0$ and $Q(x_0,\e_k) \subset \Omega$, we have
\begin{align*}
\mu^c(x_0) &= \lim_{k \to +\infty}\frac{\mu (Q(x_0,\e_k))}{|E^cu|(Q(x_0,\e_k))} \\
&=  \lim_{k,n}\left[\frac{1}{|E^cu|(Q(x_0,\e_k))}\int_{Q(x_0,\e_k)}
f(\chi_n(x),\mathcal E u_n(x))\,dx  +  |D \chi_n|(Q(x_0,\e_k))\right]
\end{align*}
where $\chi_n \in BV(Q(x_0,\e_k);\{0,1\})$, $\chi_n \to \chi$ in $L^1(Q(x_0,\e_k);\{0,1\})$, $u_n \in W^{1,1}(Q(x_0,\e_k);\Rb^N)$, $u_n \to u$ in $L^1(Q(x_0,\e_k);\Rb^N)$.
Taking into account that we are searching for a lower bound for $\mu^c(x_0)$,
we neglect the perimeter term $|D \chi_n|(Q(x_0,\e_k))$ and obtain
\begin{align}\label{lbmuc}
\mu^c(x_0) & \geq
\liminf_{k,n}\frac{1}{|E^cu|(Q(x_0,\e_k))}
\int_{Q(x_0,\e_k)}f(\chi_n(x),\mathcal E u_n(x))\,dx  \\
 & \geq
\liminf_{k,n}\frac{1}{|E^cu|(Q(x_0,\e_k))}
\int_{Q(x_0,\e_k)}f_1(\mathcal E u_n(x))\,dx \nonumber\\
& \hspace{3cm} + \liminf_{k,n}\frac{1}{|E^cu|(Q(x_0,\e_k))}
\int_{Q(x_0,\e_k)}f(\chi_n(x),\mathcal E u_n(x)) - f(1,\mathcal E u_n(x))\,dx \nonumber\\
& \geq \liminf_k \frac{\mathcal F_1(u;Q(x_0,\e_k))}{|E^cu|(Q(x_0,\e_k))}+ \liminf_{k,n}I_{k,n} \nonumber\\
& \geq \frac{d \mathcal F_{1}(u;\cdot)}{d|E^c u|}(x_0) + \liminf_{k,n}I_{k,n}
 \nonumber\\
& \geq  f^\infty \Big(1,\frac{d E^c u}{d|E^c u|}(x_0)\Big) + \liminf_{k,n}I_{k,n}
\label{lbmuc2}
\end{align}
where
$$I_{k,n} = \frac{1}{|E^cu|(Q(x_0,\e_k))}
\int_{Q(x_0,\e_k)}f(\chi_n(x),\mathcal E u_n(x)) - f(1,\mathcal E u_n(x))\,dx.$$
It remains to estimate this term. Changing variables we get
\begin{align}\label{Ikn}
\left|I_{k,n}\right| &= \left|\frac{\e_k^N}{|E^cu|(Q(x_0,\e_k))}
\int_Q f(\chi_n(x_0+\e_k y),\mathcal E u_n(x_0+\e_k y)) - 
f(1, \mathcal E u_n(x_0+\e_k y))\,dy\right| \nonumber \\
&= \left|\delta_k
\int_Q f(\chi_{n,k}(y) + 1,\mathcal E u_{n,k}(y)) - 
f(1, \mathcal E u_{n,k}(y))\,dy\right|
\end{align}
where 
$$\displaystyle \delta_k : = \frac{\e_k^N}{|E^cu|(Q(x_0,\e_k))}, \; \;
\chi_{n,k}(y) := \chi_n(x_0+\e_k y) - 1, \; \; 
u_{n,k}(y) := \frac{u_n(x_0+\e_k y) - u(x_0)}{\e_k}.$$
By \eqref{chiad} it follows that 
$\displaystyle \lim_{k,n}\|\chi_{n,k}\|_{L^1(Q)} = 0$ (see \eqref{chindeltalim})
and $\displaystyle \lim_k \delta_k = 0$.
Thus, using also \eqref{G}, we have from \eqref{Ikn}
\begin{align}\label{Ikn2}
\liminf_{k,n}\left|I_{k,n}\right| &\leq \limsup_{k,n} \delta_k 
\int_Q\left | f(\chi_{n,k}(y) + 1,\mathcal E u_{n,k}(y)) - 
f(1, \mathcal E u_{n,k}(y))\right|\,dy \nonumber \\
&\leq \limsup_{k,n} C \delta_k 
\int_Q |\chi_{n,k}(y)|\big(1 + |\mathcal E u_{n,k}(y)|\big) \, dy \nonumber \\
& = \limsup_{k,n} C \delta_k 
\int_Q |\chi_{n,k}(y)| \, |\mathcal E u_{n,k}(y)| \, dy.
\end{align}
From the growth condition from below on $f$, \eqref{lbmuc} and \eqref{Cdens} we conclude that
\begin{align}
\limsup_{k,n} C \delta_k 
\int_Q |\chi_{n,k}(y)| \, |\mathcal E u_{n,k}(y)| \, dy &\leq
\limsup_{k,n} \frac{C}{|E^cu|(Q(x_0,\e_k))}
\int_{Q(x_0,\e_k)}|\mathcal E u_{n}(x)| \, dx \nonumber \\
&\leq \limsup_{k,n} \frac{C}{|E^cu|(Q(x_0,\e_k))}
\int_{Q(x_0,\e_k)}f(\chi_n(x),\mathcal E u_{n}(x)) \, dx \nonumber \\
&\leq  C \mu^c(x_0) < + \infty. \nonumber
\end{align}
Using a diagonalisation argument, let $\chi_k := \chi_{n(k),k}$, 
$u_k := u_{n(k),k}$ be such that $\chi_k \to 0$ in $L^1(Q)$ and
\begin{equation}\label{Ikn3}
\limsup_{k,n}C \delta_k 
\int_Q |\chi_{n,k}(y)| \, |\mathcal E u_{n,k}(y)| \, dy =
\lim_{k}C \delta_k 
\int_Q |\chi_{k}(y)| \, |\mathcal E u_{k}(y)| \, dy < + \infty.
\end{equation}
Therefore, the sequence $\{\delta_k \chi_k \, \mathcal E u_{k}\}$ is bounded in $L^1(Q;\Rb^{N \times N}_s)$
so, by the Biting Lemma, there exists a subsequence (not relabeled) and there exist
sets $D_r \subset Q$ such that 
$\displaystyle \lim_{r \to + \infty}{\mathcal L}^N(D_r) = 0$ and the sequence 
$\{\delta_k \chi_k \, \mathcal Eu_k\}$ is equiintegrable in $Q\setminus D_r$, 
for any $r \in \mathbb N$. Following the reasoning in the proof of 
Proposition~\ref{lbbulk} (see \eqref{77}), for any $j \in \Nb$ there exist $k(j), r(j) \in \Nb$ such that
\begin{equation}\label{Ikn4}
\delta_k(j)
\int_{D_{r(j)}} |\chi_{k(j)}(y)| \, |\mathcal E u_{k(j)}(y)| \, dy \leq \frac{1}{j}.
\end{equation}
The fact that $\chi_{k(j)} \to 0$, as $j \to + \infty$, in $L^1(Q)$ and the equiintegrability of $\{\delta_{k(j)} \chi_{k(j)} \, \mathcal Eu_{k(j)}\}$ in 
$Q\setminus D_{r(j)}$ ensures that, for any $\e > 0$, 
\begin{equation}\label{Ikn5}
\delta_k(j)
\int_{Q \setminus D_{r(j)}} |\chi_{k(j)}(y)| \, |\mathcal E u_{k(j)}(y)| \, dy < \e,
\end{equation}
provided $j$ is large enough (see the argument used to obtain \eqref{11}).
Hence, from \eqref{lbmuc2}, \eqref{Ikn2}, \eqref{Ikn3}, \eqref{Ikn4} and \eqref{Ikn5}
we conclude that
$$\mu^c(x_0) \geq f^\infty \Big(1,\frac{d E^c u}{d|E^c u|}(x_0)\Big),$$
which completes the proof.
\end{proof}

\begin{proposition}
Let $u \in BD(\Omega)$, $\chi \in BV(\Omega;\{0,1\})$ and let $W_0$ and $W_1$ be continuous functions satisfying \eqref{growth}. 
Assume that $f$ given by \eqref{density} is symmetric quasiconvex.
Then, for $|E^cu|$ a.e. $x_0 \in \O$,
\begin{equation}\label{ubCantor}
\mu^c(x_0) = \dfrac{d\mathcal F(\chi, u; \cdot)}{d |E^cu|}(x_0) \leq f^{\infty}\left(\chi(x_0), \frac{d E^c u}{d |E^c u|}(x_0)\right).
\end{equation} 
\end{proposition}
\begin{proof}
Let $x_0 \in \Omega$ be a point satisfying \eqref{Cdens}, \eqref{uac}, \eqref{chiad} (which hold for $|E^cu|$ a.e. $x \in \Omega$, as observed in the proof of Proposition \ref{lbCantor}) and, in addition,
\begin{equation}\label{Dchi}
\lim_{\e \to 0^+}\frac{|D\chi|(Q(x_0,\e))}{|E^cu|(Q(x_0,\e))} = 0,
\end{equation}
 
Assuming, once again, that $\chi(x_0) = 1$, we also require that $x_0$ satisfies \eqref{CdensFi}. 
Choosing the sequence $\e_k \to 0^+$ in such a way that 
$\mu(\partial Q(x_0,\e_k)) = 0$ and $Q(x_0,\e_k) \subset \Omega$, 
let $u_n \in W^{1,1}(Q(x_0,\e_k);\Rb^N)$ be such that $u_n \to u$ in $L^1(Q(x_0,\e_k);\Rb^N)$ and
\begin{equation}\label{F1Cdens}
\dfrac{d\mathcal F_1(u; \cdot)}{d |E^cu|}(x_0) =
\lim_{k \to +\infty}\frac{\mathcal F_1(u;Q(x_0,\e_k))}{|E^cu|(Q(x_0,\e_k))} 
= \lim_{k,n}\frac{1}{|E^cu|(Q(x_0,\e_k))}
\int_{Q(x_0,\e_k)}f_1(\mathcal E u_n(x)) \, dx.
\end{equation}
Then, as the constant sequence $\chi_n = \chi$ is admissible for 
$\mathcal F(\chi,u;Q(x_0,\e_k))$, from \eqref{Dchi}, \eqref{F1Cdens} and \eqref{CdensFi} with $i=1$, it follows that
\begin{align*}
\mu^c(x_0) &= \lim_{k \to +\infty}
\frac{\mathcal F(\chi,u;Q(x_0,\e_k))}{|E^cu|(Q(x_0,\e_k))} \\
&\leq \liminf_{k,n}\left[\frac{1}{|E^cu|(Q(x_0,\e_k))}\int_{Q(x_0,\e_k)}
f(\chi(x),\mathcal E u_n(x))\,dx  +  |D \chi|(Q(x_0,\e_k))\right] \\
&\leq \lim_{k,n}\frac{1}{|E^cu|(Q(x_0,\e_k))}\int_{Q(x_0,\e_k)}
f(1,\mathcal E u_n(x))\,dx  \\
&\hspace{3cm} + \limsup_{k,n}\frac{1}{|E^cu|(Q(x_0,\e_k))}\int_{Q(x_0,\e_k)}
f(\chi(x),\mathcal E u_n(x)) - f(1,\mathcal E u_n(x)) \,dx  \\
& = f^{\infty}\left(\chi(x_0), \frac{d E^c u}{d |E^c u|}(x_0)\right) + \limsup_{k,n}\frac{1}{|E^cu|(Q(x_0,\e_k))}\int_{Q(x_0,\e_k)}
\hspace{-0,4cm}f(\chi(x),\mathcal E u_n(x)) - f(1,\mathcal E u_n(x)) \,dx.
\end{align*}
The same argument used in the proof of Proposition~\ref{lbCantor}, now applied to the sequences 
$$\chi_{k}(y) = \chi(x_0 + \e_ky) - 1, \; \;   
u_{n,k}(y) := \frac{u_n(x_0+\e_k y) - u(x_0)}{\e_k},$$ yields
$$\limsup_{k,n}\frac{1}{|E^cu|(Q(x_0,\e_k))}\int_{Q(x_0,\e_k)}
f(\chi(x),\mathcal E u_n(x)) - f(1,\mathcal E u_n(x)) \,dx = 0$$
from which the conclusion follows.
\end{proof}

\subsection{The Surface Term}\label{surface}

Given $x_0 \in J_\chi \cup J_u$ we denote by $\nu(x_0)$ the vector $\nu_u(x_0)$, 
if $x_0 \in J_u \setminus J_\chi$, whereas $\nu(x_0) := \nu_{\chi}(x_0)$ if
$x_0 \in J_\chi \setminus J_u$, these vectors are well defined as Borel measurable functions for $\cH^{N-1}$ a.e. $x_0 \in J_\chi \cup J_u$. Due to the rectifiability of both $J_\chi$ and $J_u$ (cf. \cite[Theorems 3.77 and 3.78]{AFP} and \cite[Proposition 3.5 and Remark 3.6]{ACDM}),
for $\cH^{N-1}$ a.e. $x_0 \in J_\chi \cap J_u$ we may select 
$\nu(x_0) := \nu_{\chi}(x_0) = \nu_u(x_0)$ where the orientation of $\nu_{\chi}(x_0)$ is
chosen so that $\chi^+(x_0) = 1$, $\chi^-(x_0) = 0$ and then $u^+(x_0)$ and $u^-(x_0)$ are selected according to this orientation.

Thus, in the sequel for $\mathcal H^{N-1}$ a.e. $x_0 \in J_{\chi} \cup J_u$,  the vector $\nu(x_0)$ is defined according to the above considerations.

Given that $\mathcal H^{N-1}(S_\chi \setminus J_\chi) = 0$ and that all points in $\Omega \setminus S_\chi$ are Lebesgue points of $\chi$, in what follows we take  
$\chi^+(x_0) = \chi^-(x_0) = \widetilde\chi(x_0)$ for
$\mathcal H^{N-1}$- a.e. $x_0 \in J_u \setminus J_\chi$, where $\tilde v$ denotes the precise representative of a field $v$ in $BV$, cf. Section \ref{BV}. On the other hand, for a $BD$ function $u$ it is not known whether 
$\mathcal H^{N-1}(S_u \setminus J_u) = 0$. However, given that all points in 
$\Omega \setminus S_u$ are Lebesgue points of $u$ and that, by \cite[Remark 6.3]{ACDM} and the $\mathcal H^{N-1}$ rectifiability of $J_\chi$, 
$\mathcal H^{N-1}(S_u \setminus J_u)\cap J_\chi)=0$, we may consider 
$u^+(x_0) = u^-(x_0) = \widetilde u(x_0)$ for $\mathcal H^{N-1}$ a.e.
$x_0 \in J_\chi \setminus J_u$, where $\tilde v$ denotes the Lebesgue representative of a field $v$ in $BD$ (cf. \cite[page 206]{ACDM}), see also \cite{B}.

In order to describe $\mu^j$ we will follow the ideas of the global method for relaxation introduced in \cite{BFM} (see also \cite{BFT} and \cite{CFVG}), the sequential characterisation of $K(a,b,c,d, \nu)$, obtained in Proposition \ref{Ktilde}, will also be used.

Given $u \in BD(\O)$, $\chi \in BV(\O;\{0,1\})$ and $V \in \mathcal O_{\infty}(\O)$
we define
\begin{equation}\label{m}
	m(\chi,u;V) := \inf\left\{\mathcal F (\theta,v;V) : 
\theta \in  BV(\O;\{0,1\}), v \in BD(\O), \theta = \chi \mbox{ on } \partial V,
v = u \mbox{ on } \partial V\right\}.
\end{equation}

Our goal is to show the following result.

\begin{proposition}\label{mg}
Let $f$ be given by \eqref{density}, where $W_0$ and $W_1$ are continuous functions satisfying \eqref{growth}.
Given $u \in SBD(\Omega)$ and $\chi \in BV(\Omega;\{0,1\})$,  we have
$$\mathcal{F}( \chi, u; V \cap (J_\chi \cup J_u)) =
\int_{V \cap (J_\chi \cup J_u)} 
g(x,\chi^+(x),\chi^-(x),u^+(x),u^-(x),\nu(x))\, d\cH^{N-1}(x),$$
where 
\begin{equation}\label{defg}
g(x_0,a,b,c,d,\nu) := \limsup_{\varepsilon \to 0^+}
\frac{m(\chi_{a,b,\nu}(\cdot - x_0),u_{c,d,\nu}(\cdot - x_0);Q_\nu(x_0,\e))}
{\e^{N-1}}
\end{equation}
and $\chi_{a,b,\nu}$, $u_{c,d,\nu}$ were defined in \eqref{targets}.
\end{proposition}

The proof of the above proposition relies on a series of auxiliary results, based on Lemmas 3.1, 3.3 and 3.5 in \cite{BFM} and which were adapted to the $BD$ case in \cite{BFT}[Lemmas 3.10, 3.11 and 3.12]. The properties of $\mathcal F(\chi,u;A)$ established in Proposition \ref{firstprop}, and the fact that 
$\mathcal F(\chi,u;\cdot)$ is a Radon measure, ensure that we can apply the reasoning given in their respective proofs.

\begin{lemma}\label{BFT310}
Let $f$ be given by \eqref{density}, where $W_0$ and $W_1$ are continuous functions satisfying \eqref{growth}.
Then there exists a positive constant $C$ such that
$$|m(\chi_1,u_1;V) - m(\chi_2,u_2;V)| \leq 
C\left[\int_{\partial V}|{\rm tr } \, \chi_1(x) - {\rm tr}\,\chi_2(x)| +  
|{\rm tr}\,u_1(x) - {\rm tr}\,u_2(x)| \, d\cH^{N-1}(x) \right],$$
for every $\chi_1,\chi_2 \in BV(\O;\{0,1\})$, $u_1,u_2 \in BD(\O)$ and any
$V \in \mathcal O_{\infty}(\O).$
\end{lemma}
\begin{proof}
The proof follows that of Lemma 3.10 in \cite{BFT}. Given $\delta > 0$ let
$V_\delta := \{x \in V : {\rm dist}(x,\partial V) > \delta\}$ and select
$\theta \in BV(\O;\{0,1\})$ and $v \in BD(\O)$ such that $\theta = \chi_2$ and
$v = u_2$ on $\partial V$. Now define $\theta_\delta \in BV(\O;\{0,1\})$ and 
$v_\delta \in BD(\O)$ by
$$\theta_\delta : = \begin{cases}
\theta, & {\rm in} \; V_\delta\\
\chi_1, & {\rm in} \; \O \setminus V_\delta
\end{cases} \; \; \; {\rm and} \; \; \; 
v_\delta : = \begin{cases}
v, & {\rm in} \; V_\delta\\
u_1, & {\rm in} \; \O \setminus V_\delta.
\end{cases}
$$
The definition of $m(\cdot,\cdot;\cdot)$ and the additivity and locality of 
$\mathcal F(\cdot,\cdot;\cdot)$, as well as the inequality from above in Proposition
\ref{firstprop} $i)$, lead to the conclusion.
\end{proof}

Fixing $\chi \in BV(\O;\{0,1\})$, $u \in BD(\O)$ and 
$\nu \in S^{N-1}$, we define $\lambda := \LL^N + |E^s u| + |D\chi|$ and, following \cite{BFM}, we let
$$\mathcal O^*:= \left\{Q_\nu(x,\e) : x \in \O,\,  \e > 0\right\}$$
and, for $\delta > 0$ and $V \in \mathcal O(\O)$, set
\begin{align*}
m^\delta(\chi,u;V) &:= \inf\Big\{\sum_{i = 1}^{+\infty}m(\chi,u;Q_i) : 
Q_i \in \mathcal O^*, Q_i \cap Q_j = \emptyset \; {\rm if} \; i \neq j,\\
& \hspace{4cm} Q_i \subset V, {\rm diam} \, Q_i < \delta, 
\lambda\left(V \setminus \displaystyle\bigcup_{i=1}^{+\infty}Q_i\right) = 0\Big\}.
\end{align*}
Clearly, $\delta \mapsto m^\delta(\chi,u;V)$ is a decreasing function, so we define
$$m^*(\chi,u;V) := \sup \left\{m^\delta(\chi,u;V) : \delta > 0\right\}
= \lim_{\delta \to 0^+}m^\delta(\chi,u;V).$$

\begin{lemma}\label{BFT311}
Let $f$ be given by \eqref{density}, where $W_0$ and $W_1$ are continuous functions satisfying \eqref{growth}.
Given  $\chi \in BV(\O;\{0,1\})$, $u \in BD(\O)$, we have
$$\mathcal F(\chi,u;V) = m^*(\chi,u;V), \; \mbox{ for every } V \in \mathcal O(\O).$$
\end{lemma}
\begin{proof}
The inequality 
$$m^*(\chi,u;V) \leq \mathcal F(\chi,u;V)$$ is an immediate consequence
of the fact that $m(\chi,u;Q_i) \leq \mathcal F(\chi,u;Q_i)$ and that 
$\mathcal F(\chi,u;\cdot)$ is a Radon measure.

The proof of the reverse inequality relies on the lower semicontinuity of 
$\mathcal F(\cdot,\cdot;V)$ obtained in 
Proposition \ref{firstprop} $iv)$ and on the definitions of $m^\delta(\chi,u;V)$, $m(\chi,u;V)$ and $m^*(\chi,u;V)$. Indeed, fixing $\delta > 0$, we consider $(Q_i^\delta)$ an admissible family for $m^\delta(\chi,u;V)$ such that,
letting $\displaystyle N^\delta := V \setminus \displaystyle\cup_{i=1}^{+\infty}Q_i^\delta$,
$$\sum_{i =1}^{+\infty}m(\chi,u;Q_i^\delta) < m^\delta(\chi,u;V) + \delta 
\; \mbox{ and } \; \lambda(N^\delta) = 0,$$
and we now let $\theta_i^\delta \in BV(\O;\{0,1\})$ and $v_i^\delta \in BD(\O)$
be such that $\theta_i^\delta = \chi$ on $\partial Q_i^\delta$, $v_i^\delta = u$ on
$\partial Q_i^\delta$ and
$$\mathcal F(\theta_i^\delta,v_i^\delta;Q_i^\delta) \leq 
m(\chi,u;Q_i^\delta) + \delta \LL^N(Q_i^\delta).$$
Setting 
$\displaystyle N_0^\delta := \O\setminus\displaystyle\cup_{i=1}^{+\infty}Q_i^\delta$,
we define 
$$\theta^\delta := \sum_{i =1}^{+\infty}\theta_i^\delta \,\chi_{Q_i^\delta} 
+ \chi \,\chi_{N_0^\delta} \;\;\; \mbox{ and } \; \;\;
v^\delta := \sum_{i =1}^{+\infty}v_i^\delta \,\chi_{Q_i^\delta} 
+ u \,\chi_{N_0^\delta}.$$
Following the computations in the proof of \cite[Lemma 3.11]{BFT}, we may show that $\theta^\delta \in BV(\O;\{0,1\})$, $v^\delta \in BD(\O)$, 
$\theta^\delta \to \chi$  in $L^1(V;\{0,1\})$ and
$v^\delta \to u$ in $L^1(V;\Rb^N)$, as $\delta \to 0^+$, and also
$$\mathcal F(\theta^\delta,v^\delta;N^\delta) \leq C \lambda(N^\delta) = 0.$$ 
Using the additivity of $\mathcal F(\theta^\delta,v^\delta;\cdot)$ we have
\begin{align*}
\mathcal F(\theta^\delta,v^\delta;V) &= \sum_{i =1}^{+\infty}
\mathcal F(\theta_i^\delta,v_i^\delta;Q_i^\delta) + 
\mathcal F(\theta^\delta,v^\delta;N^\delta)\\
&\leq \sum_{i =1}^{+\infty}m(\chi,u;Q_i^\delta) + \delta \LL^N(V)
\leq m^\delta(\chi,u;V) + \delta + \delta \LL^N(V),
\end{align*}
so that the lower semicontinuity of $\mathcal F(\cdot,\cdot;V)$ yields
\begin{align*}
\mathcal F(\chi,u;V) &\leq \liminf_{\delta \to 0^+}
\mathcal F(\theta^\delta,v^\delta;V) \\
&\leq \liminf_{\delta \to 0^+}
\left( m^\delta(\chi,u;V) + \delta + \delta \LL^N(V)\right) = m^*(\chi,u;V)
\end{align*}
and this completes the proof.
\end{proof}

Finally, a straightforward adaptation of \cite[Lemma 3.5]{BFM} leads to the following
result.

\begin{lemma}\label{BFT312}
Let $f$ be given by \eqref{density}, where $W_0$ and $W_1$ are continuous functions satisfying \eqref{growth}.
Given  $\chi \in BV(\O;\{0,1\})$, $u \in BD(\O)$, we have
$$\lim_{\e \to 0^+}\frac{\mathcal F(\chi,u;Q_\nu(x_0,\e))}{\lambda(Q_\nu(x_0,\e))} = \lim_{\e \to 0^+}\frac{m(\chi,u;Q_\nu(x_0,\e))}{\lambda(Q_\nu(x_0,\e))},$$
for $\lambda$ a.e. $x_0 \in \O$ and for every $\nu \in S^{N-1}$.
\end{lemma}

We now proceed with the proof of Proposition \ref{mg}.

\begin{proof}[Proof of Proposition \ref{mg}]
In the sequel, for simplicity of notation, we will write $\nu = \nu(x_0)$.

Let $x_0 \in \O \cap (J_\chi \cup J_u)$ be a point satisfying 
\begin{equation}\label{chitilde}
\lim_{\varepsilon \rightarrow 0^+}\frac{1}{\e^N}
\int_{Q_{\nu}(x_0,\e)}|\chi(x) - \widetilde{\chi}(x_0)| \, dx = 0, 
\; {\rm if} \;x_0 \in \Omega \setminus J_{\chi},
\end{equation}
\begin{equation}\label{chipm}
\lim_{\varepsilon \rightarrow 0^+}\frac{1}{\e^N}
\int_{Q^+_{\nu}(x_0,\e)}\hspace{-0,2cm}|\chi(x) - {\chi}^+(x_0)| \, dx = 
\lim_{\varepsilon \rightarrow 0^+}\frac{1}{\e^N}
\int_{Q^-_{\nu}(x_0,\e)}\hspace{-0,2cm}|\chi(x) - {\chi}^-(x_0)| \, dx =0, 
\; {\rm if} \;x_0 \in \Omega \cap J_{\chi},
\end{equation}
\begin{equation}\label{utilde}
\lim_{\varepsilon \rightarrow 0^+}\frac{1}{\e^N}
\int_{Q_{\nu}(x_0,\e)}|u(x) - \widetilde{u}(x_0)| \, dx = 0, 
\; {\rm if} \;x_0 \in \Omega \setminus J_{u},
\end{equation}
\begin{equation}\label{upm}
\lim_{\varepsilon \rightarrow 0^+}\frac{1}{\e^N}
\int_{Q^+_{\nu}(x_0,\e)}\hspace{-0,2cm}|u(x) - {u}^+(x_0)| \, dx = 
\lim_{\varepsilon \rightarrow 0^+}\frac{1}{\e^N}
\int_{Q^-_{\nu}(x_0,\e)}\hspace{-0,2cm}|u(x) - {u}^-(x_0)| \, dx =0, 
\; {\rm if} \;x_0 \in \Omega \cap J_{u},
\end{equation}
where 
$$Q^{\pm}_{\nu}(x_0,\e) = 
\left\{x \in Q_{\nu}(x_0,\e) : (x - x_0) \cdot (\pm \nu) > 0\right\},$$
and
\begin{equation}\label{muj}
\mu^j(x_0) = \lim_{\varepsilon \rightarrow 0^+}
\frac{\mathcal F(\chi,u;Q_{\nu}(x_0,\e))}
{\mathcal H^{N-1} \lfloor (J_\chi \cup J_u)(Q_{\nu}(x_0,\e))} =
\lim_{\varepsilon \rightarrow 0^+}\frac{1}{\e^{N-1}}\int_{Q_{\nu}(x_0,\e)}d\mu(x) \;
\mbox{ exists and is finite}.
\end{equation}
In view of the considerations made at the beginning of this subsection, these properties hold for $\mathcal H^{N-1}$ a.e. 
$x_0 \in \Omega\cap (J_\chi \cup J_u)$.
Furthermore, we require that $x_0$ also satisfies
\begin{equation}\label{symjump}
\lim_{\e \to 0^+}\frac{1}{\e^{N-1}}|Eu|(Q_{\nu}(x_0,\e)) = 
|([u]\odot \nu)(x_0)| = |Eu_0|(Q_\nu)
\end{equation}
and
\begin{equation}\label{perchi0}
\lim_{\e \to 0^+}\frac{1}{\e^{N-1}}|D\chi|(Q_{\nu}(x_0,\e)) = 1 = |D\chi_0|(Q_\nu),
\end{equation}
where we are denoting by $\chi_0$ and $u_0$ the functions given by 
\eqref{targets} with $\nu = \nu(x_0)$ and $a=\chi^+(x_0)$, 
$b=\chi^-(x_0)$, $c=u^+(x_0)$ and  $d=u^-(x_0)$.
Letting $\sigma := \cH^{N-1}\lfloor (J_\chi \cup J_u)$, by Lemma \ref{BFT312}
it follows that, for $\sigma$ a.e. $x_0 \in \O$,
\begin{equation}\label{msigma}
\frac{d\mathcal F(\chi,u;\cdot)}{d \sigma}(x_0) =
\lim_{\varepsilon \to 0^+}
\frac{\mathcal F(\chi,u;Q_\nu(x_0,\e))}{\sigma(Q_\nu(x_0,\e))} =
\lim_{\varepsilon \to 0^+}
\frac{m(\chi,u;Q_\nu(x_0,\e))}{\sigma(Q_\nu(x_0,\e))}.
\end{equation}
Let $\chi_{\e}: Q_\nu \to \{0,1\}$ and $u_\e: Q_\nu \to \Rb^N$ be defined by
$\chi_{\e}(y) := \chi(x_0 + \e y)$, $u_\e(y) :=u(x_0 + \e y)$. Properties \eqref{chitilde} or \eqref{chipm}, and \eqref{utilde} or \eqref{upm}, respectively,
guarantee that $\chi_{\e} \to \chi_0$ in $L^1(Q_\nu;\{0,1\})$ and
$u_\e \to u_0$ in $L^1(Q_\nu;\Rb^N)$. On the other hand, by \eqref{symjump} and
\eqref{perchi0} we have
$$\lim_{\e \to 0^+}|Eu_\e|(Q_\nu) = 
\lim_{\e \to 0^+}\frac{1}{\e^{N-1}}|Eu|(Q_\nu(x_0,\e)) 
= |([u]\odot \nu)(x_0)| = |Eu_0|(Q_\nu)$$
and
$$\lim_{\e \to 0^+}|D\chi_\e|(Q_\nu) = 
\lim_{\e \to 0^+}\frac{1}{\e^{N-1}}|D\chi|(Q_\nu(x_0,\e))  = |D\chi_0|(Q_\nu).$$
Due to the continuity of the trace operator with respect to the intermediate topology
we conclude that
\begin{align}\label{trcont}
& \lim_{\e \to 0^+}\frac{1}{\e^{N-1}}\int_{\partial Q_{\nu}(x_0,\e)}
|{\rm tr } \, \chi(x) - {\rm tr}\,\chi_0(x-x_0)| +  
|{\rm tr}\,u(x) - {\rm tr}\,u_0(x-x_0)| \, d\cH^{N-1}(x) \nonumber \\
& = \lim_{\e \to 0^+}\int_{\partial Q_{\nu}}
|{\rm tr } \, \chi_\e(y) - {\rm tr}\,\chi_0(y)| +  
|{\rm tr}\,u_\e(y) - {\rm tr}\,u_0(y)| \, d\cH^{N-1}(y) = 0.
\end{align}
Hence, from \eqref{msigma}, \eqref{muj}, Lemma \ref{BFT310} and \eqref{trcont}, we obtain
\begin{align*}
&\frac{d\mathcal F(\chi,u;\cdot)}{d \sigma}(x_0) = 
\lim_{\varepsilon \to 0^+}
\frac{m(\chi,u;Q_\nu(x_0,\e))}{\sigma(Q_\nu(x_0,\e))} \\
& = \lim_{\varepsilon \to 0^+}\frac{m(\chi,u;Q_\nu(x_0,\e))- 
m(\chi_0(\cdot - x_0),u_0(\cdot - x_0);Q_\nu(x_0,\e)) 
+ m(\chi_0(\cdot - x_0),u_0(\cdot - x_0);Q_\nu(x_0,\e))}{\e^{N-1}}\\
& = \lim_{\varepsilon \to 0^+}
\frac{m(\chi_0(\cdot - x_0),u_0(\cdot - x_0);Q_\nu(x_0,\e))}{\e^{N-1}}
\end{align*}
and, therefore,
\begin{align*}
\mathcal{F}( \chi, u; V \cap (J_\chi \cup J_u)) &=
\int_{V \cap (J_\chi \cup J_u)} 
\frac{d\mathcal F(\chi,u;\cdot)}{d \sigma}(x) \, d\sigma(x) \\
&=\int_{V \cap (J_\chi \cup J_u)} 
g(x,\chi^+(x),\chi^-(x),u^+(x),u^-(x),\nu(x))\, d\cH^{N-1}(x).
\end{align*}
\end{proof}

In the final two propositions we will show that, under assumption \eqref{finfty}, the surface energy density $g(x_0,a,b,c,d,\nu)$
may be more explicitly characterised. For this purpose we need an additional lemma
which states that more regular functions can be considered in the definition of the
Dirichlet functional
$m(\chi,u;V)$ in \eqref{m}. In what follows, for
$u \in BD(\O)$, $\chi \in BV(\O;\{0,1\})$ and $V \in \mathcal O_{\infty}(\O)$
we define
$$m_0(\chi,u;V) := \inf\left\{F (\theta,v;V) : 
\theta \in  BV(\O;\{0,1\}), v \in W^{1,1}(\O;\Rb^N), \theta = \chi \mbox{ on } \partial V,
v = u \mbox{ on } \partial V\right\}.$$

\begin{lemma}\label{BFT314}
Let $f$ be given by \eqref{density}, where $W_0$ and $W_1$ are continuous functions satisfying \eqref{growth}.
Given  $\chi \in BV(\O;\{0,1\})$, $u \in BD(\O)$, we have
$$m(\chi,u;V) = m_0(\chi,u;V), \; \mbox{ for every } V \in \mathcal O_{\infty}(\O).$$
\end{lemma}
\begin{proof}
The inequality $m(\chi,u;V) \leq m_0(\chi,u;V)$ is clear since, given any
$\theta \in  BV(\O;\{0,1\})$ such that $\theta = \chi \mbox{ on } \partial V$ and any $v \in W^{1,1}(\O;\Rb^N)$ such that 
$v = u \mbox{ on } \partial V$, we have
$$m(\chi,u;V) \leq \mathcal F(\theta,v;V) \leq F(\theta,v;V).$$
To show the reverse inequality, we fix $\e > 0$ and let $\theta \in  BV(\O;\{0,1\}), 
v \in BD(\O)$ be such that 
$\theta = \chi \mbox{ on } \partial V, v = u \mbox{ on } \partial V$ and
$$m(\chi,u;V) + \e \geq \mathcal F(\theta,v;V).$$
By Proposition~\ref{firstprop} $ii)$, let $\chi_n \in  BV(\O;\{0,1\}), u_n \in W^{1,1}(\O;\Rb^N)$ satisfy $\chi_n \to \theta$ in $L^1(\O;\{0,1\})$, 
$u_n \to v$ in $L^1(\O;\Rb^N)$ and
$$\mathcal F(\theta,v;V) = \lim_{n \to + \infty}F(\chi_n,u_n;V).$$
Theorem~\ref{densitysmooth} ensures the existence of a sequence
$v_n \in W^{1,1}(\O;\Rb^N)$ such that $v_n \to v$ in $L^1(\O;\Rb^N)$,
$v_n = v = u \mbox{ on } \partial V$ and $|Ev_n|(V) \to |Ev|(V)$.
We now apply Proposition~\ref{newslicing} to conclude that there exists a subsequence
$\{v_{n_k}\}$ of $\{v_n\}$ and there exist sequences $w_k \in W^{1,1}(\O;\Rb^N)$,
$\eta_k \in BV(\O;\{0,1\})$ verifying $w_k = v_{n_k} = u$ on $\partial V$,
$\eta_k = \theta = \chi$ on $\partial V$ and
$$\limsup_{k \to +\infty}F(\eta_k,w_k;V) \leq 
\liminf_{n \to +\infty}F(\chi_n,u_n;V).$$
Therefore,
$$m_0(\chi,u;V) \leq \limsup_{k \to +\infty}F(\eta_k,w_k;V) \leq 
\liminf_{n \to +\infty}F(\chi_n,u_n;V) = \mathcal F(\theta,v;V) \leq 
m(\chi,u;V) + \e,$$
so the desired inequality follows by letting $\e \to 0^+.$
\end{proof}

\begin{proposition}
Let $f$ be given by \eqref{density}, where $W_0$ and $W_1$ are continuous functions satisfying \eqref{growth}. 
Assume that $f$ is symmetric quasiconvex and that \eqref{finfty} holds.
Given $u \in BD(\Omega)$ and $\chi \in BV(\Omega;\{0,1\})$,  
for $\mathcal H^{N-1}$ a.e. $x_0 \in \Omega \cap (J_\chi \cup J_u)$, we have
$$g(x_0,\chi^+(x_0),\chi^-(x_0),u^+(x_0),u^-(x_0),\nu(x_0)) \geq
K(\chi^+(x_0),\chi^-(x_0),u^+(x_0),u^-(x_0),\nu(x_0)),$$
where $\chi^+(x_0) = \chi^-(x_0) = \widetilde\chi(x_0)$ if 
$x_0 \in J_u \setminus J_\chi$ and $u^+(x_0) = u^-(x_0) = \widetilde u(x_0)$ if
$x_0 \in J_\chi \setminus J_u$, and $K$ is given by \eqref{K}.
\end{proposition}
\begin{proof}

As before, for simplicity of notation, we write $\nu = \nu(x_0)$.

By Lemma~\ref{BFT314} we have
\begin{align*}
&g(x_0,\chi^+(x_0),\chi^-(x_0),u^+(x_0),u^-(x_0),\nu(x_0)) \\
& \qquad =\limsup_{\e \to 0^+}\frac{1}{\e^{N-1}}
\inf\Big\{F (\theta,v;Q_\nu(x_0,\e)) : 
\theta \in  BV(\O;\{0,1\}), v \in W^{1,1}(\O;\Rb^N), \\
& \hspace{3,5cm} \theta = \chi_0(\cdot - x_0) \mbox{ on } \partial Q_\nu(x_0,\e),
v = u_0(\cdot - x_0) \mbox{ on } \partial Q_\nu(x_0,\e)\Big\},
\end{align*}
where $\chi_0$ and $u_0$ are given by \eqref{targets} with $\nu = \nu(x_0)$ and $a=\chi^+(x_0)$, 
$b=\chi^-(x_0)$, $c=u^+(x_0)$ and $d=u^-(x_0)$, respectively. Thus, for every 
$n \in \Nb$, there exist $\theta_{n,\e} \in  BV(\O;\{0,1\})$, 
$v_{n,\e} \in W^{1,1}(\O;\Rb^N)$ such that $\theta_{n,\e} = \chi_0(\cdot - x_0)$ on $\partial Q_\nu(x_0,\e)$, $v_{n,\e} = u_0(\cdot - x_0)$ on $\partial Q_\nu(x_0,\e)$
and 
\begin{align}\label{g}
& g(x_0,\chi^+(x_0),\chi^-(x_0),u^+(x_0),u^-(x_0),\nu(x_0)) + \frac{1}{n} \nonumber\\
&\hspace{2cm}\geq \limsup_{\e \to 0^+}\frac{1}{\e^{N-1}}
\left[\int_{Q_{\nu}(x_0,\e)}f(\theta_{n,\e}(x), \E v_{n,\e}(x)) \, dx + |D\theta_{n,\e}|(Q_{\nu}(x_0,\e))\right] \nonumber\\
& \hspace{2cm}= \limsup_{\varepsilon \rightarrow 0^+}
\left[\int_{Q_{\nu}}\e f(\theta_{n,\e}(x_0+\e y), \E v_{n,\e}(x_0+\e y)) \, dy + 
\int_{Q_{\nu}\cap \frac{1}{\e}(J_{\theta_{n,\e}}-x_0)}d\mathcal H^{N-1}(y)\right] \nonumber\\
& \hspace{2cm}= \limsup_{\varepsilon \rightarrow 0^+}
\left[\int_{Q_{\nu}}\e f(\chi_{n,\e}(y), \frac{1}{\e}\E u_{n,\e}(y)) \, dy + 
|D\chi_{n,\e}|(Q_{\nu})\right]  \nonumber\\
& \hspace{2cm} \geq  \liminf_{\varepsilon \rightarrow 0^+}
\left[\int_{Q_{\nu}}f^{\infty}(\chi_{n,\e}(y),\E u_{n,\e}(y)) \, dy + 
|D\chi_{n,\e}|(Q_{\nu})\right]  \nonumber\\
& \hspace{2cm} +  \liminf_{\varepsilon \rightarrow 0^+}
\int_{Q_{\nu}}\left[\e f(\chi_{n,\e}(y), \frac{1}{\e}\E u_{n,\e}(y))  -
f^{\infty}(\chi_{n,\e}(y),\E u_{n,\e}(y))\right]  dy,
\end{align}
where $\chi_{n,\e}(y) = \theta_{n,\e}(x_0 + \e y)$ and 
$u_{n,\e}(y) = v_{n,\e}(x_0+ \e y)$.
We claim that 
\begin{equation}\label{limit0}
\liminf_{\e \to 0^+}
\int_{Q_{\nu}}\left[\e f(\chi_{n,\e}(y), \frac{1}{\e}\E u_{n,\e}(y))- f^{\infty}(\chi_{n,\e}(y),\E u_{n,\e}(y))\right]  dy  = 0.
\end{equation}
If so, noticing that $(\chi_{n,\e},u_{n,\e}) \in \mathcal{A}(\chi^+(x_0),\chi^-(x_0),u^+(x_0),u^-(x_0),\nu(x_0))$, we have
from \eqref{g}, \eqref{limit0} and the definition of $K(a,b,c,d,\nu)$,
\begin{align*}
&g(x_0,\chi^+(x_0),\chi^-(x_0),u^+(x_0),u^-(x_0),\nu(x_0)) + \frac{1}{n} \\
&\hspace{3cm}
\geq \liminf_{\e \to 0^+}
\left[\int_{Q_{\nu}}f^{\infty}(\chi_{n,\e}(y),\E u_{n,\e}(y)) \, dy + 
|D\chi_{n,\e}(Q_{\nu})\right] \\
&\hspace{3cm}\geq K(\chi^+(x_0),\chi^-(x_0),u^+(x_0),u^-(x_0),\nu(x_0)),
\end{align*}
hence the result follows by letting $n \to +\infty$.

It remains to prove \eqref{limit0}. We write
\begin{align*}
& \int_{Q_{\nu}}\e f(\chi_{n,\e}(y), \frac{1}{\e}\E u_{n,\e}(y))- f^{\infty}(\chi_{n,\e}(y),\E u_{n,\e}(y)) \, dy \\
& = \int_{Q_{\nu}\cap \{\frac{1}{\e}|\E u_{n,\e}(y)| \leq L\}}\e f(\chi_{n,\e}(y), \frac{1}{\e}\E u_{n,\e}(y))- f^{\infty}(\chi_{n,\e}(y),\E u_{n,\e}(y)) \, dy \\
& + \int_{Q_{\nu}\cap \{\frac{1}{\e}|\E u_{n,\e}(y)| > L\}}\e f(\chi_{n,\e}(y), \frac{1}{\e}\E u_{n,\e}(y))- f^{\infty}(\chi_{n,\e}(y),\E u_{n,\e}(y)) \, dy
= : I_1 + I_2.
\end{align*}
By the growth hypothesis \eqref{growth} and \eqref{finftygr} we have
\begin{align*}
|I_1| &\leq \int_{Q_{\nu}\cap \{|\E u_{n,\e}(y)| \leq \e L\}}
\e \, C \left(1 + \frac{1}{\e}|\E u_{n,\e}(y)|\right) + C \, |\E u_{n,\e}(y)| \, dy \\
& \leq \int_{Q_{\nu}}\e \, C \, dy = O(\e)
\end{align*}
and, by hypothesis \eqref{finfty} with $t = \frac{1}{\e}$, H\"older's inequality and
\eqref{growth},
\begin{align*}
|I_2| &\leq \int_{Q_{\nu}\cap \{\frac{1}{\e}|\E u_{n,\e}(y)| > L\}}
\left|\e f(\chi_{n,\e}(y), \frac{1}{\e}\E u_{n,\e}(y))- 
f^{\infty}(\chi_{n,\e}(y),\E u_{n,\e}(y))\right| \, dy \\
& \leq \int_{Q_{\nu}\cap \{\frac{1}{\e}|\E u_{n,\e}(y)| > L\}}
C \, \e^{\gamma} \, |\E u_{n,\e}(y)|^{1-\gamma} \, dy \\
& \leq C \, \e^{\gamma}
\left( \int_{Q_{\nu}}|\E u_{n,\e}(y)| \, dy\right)^{1-\gamma}\\
& \leq C \, \e^{\gamma}
\left( \int_{Q_{\nu}}\e \,f(\chi_{n,\e}(y),\frac{1}{\e}\E u_{n,\e}(y)) \, dy\right)^{1-\gamma} = O(\e^{\gamma}),
\end{align*}
since the integral in the last expression is uniformly bounded by \eqref{g}. The above estimates yield \eqref{limit0} and complete the proof.
\end{proof}

\begin{proposition}
Let $f$ be given by \eqref{density}, where $W_0$ and $W_1$ are continuous functions satisfying \eqref{growth}, and assume that $f$ is symmetric quasiconvex
and that \eqref{finfty} holds.
Given $u \in BD(\Omega)$ and $\chi \in BV(\Omega;\{0,1\})$, 
for $\mathcal H^{N-1}$ a.e. $x_0 \in \Omega \cap (J_\chi \cup J_u)$ we have
$$g(x_0,\chi^+(x_0),\chi^-(x_0),u^+(x_0),u^-(x_0),\nu(x_0)) \leq
K(\chi^+(x_0),\chi^-(x_0),u^+(x_0),u^-(x_0),\nu(x_0)).$$
\end{proposition}
\begin{proof}
Using the sequential characterisation of $K$ given in Proposition \ref{Ktilde}, let 
$\chi_n \in BV(Q_\nu;\{0,1\})$, $u_n \in W^{1,1}(Q_\nu;\Rb^N)$ be such that
$\chi_n \to \chi_0$ in $L^1(Q_\nu;\{0,1\})$, $u_n \to u_0$ in $L^1(Q_\nu;\Rb^N)$
and 
$$K(\chi^+(x_0),\chi^-(x_0),u^+(x_0),u^-(x_0),\nu(x_0)) = 
\lim_{n\rightarrow+\infty}
\left[ \int_{Q_\nu} f^\infty(\chi_n(y),\E u_n(y)) \, dy + |D\chi_n|(Q_\nu)\right],$$
where $\chi_0$, $u_0$ are as in the proof of Proposition~\ref{mg}. 

For $x \in Q_\nu(x_0,\e)$, set 
$\displaystyle \theta_n(x) := \chi_n\left(\frac{x - x_0}{\e}\right)$ and 
$\displaystyle v_n(x) := u_n\left(\frac{x - x_0}{\e}\right)$. Then, changing variables and using the positive homogeneity of $f^\infty(q,\cdot)$, we have
\begin{align}\label{Kg1}
&K(\chi^+(x_0),\chi^-(x_0),u^+(x_0),u^-(x_0),\nu(x_0)) =
\lim_{n\rightarrow+\infty}
\left[ \int_{Q_\nu} f^\infty(\chi_n(y),\E u_n(y)) \, dy + |D\chi_n|(Q_\nu)\right] \nonumber \\
& \hspace{2cm}= \frac{1}{\e^{N-1}}\lim_{n\rightarrow+\infty}
\left[ \int_{Q_\nu(x_0,\e)} f^\infty(\theta_n(x),\E v_n(x)) \, dx + |D\theta_n|(Q_\nu(x_0,\e))\right] \nonumber \\
& \hspace{2cm}\geq \frac{1}{\e^{N-1}}\liminf_{n\rightarrow+\infty}
\left[ \int_{Q_\nu(x_0,\e)} f(\theta_n(x),\E v_n(x)) \, dx + |D\theta_n|(Q_\nu(x_0,\e))\right] \nonumber \\
& \hspace{2cm} + \frac{1}{\e^{N-1}}\liminf_{n\rightarrow+\infty}
\int_{Q_\nu(x_0,\e)} \Big(f^\infty(\theta_n(x),\E v_n(x)) - f(\theta_n(x),\E v_n(x))\Big)\, dx =: I_1 + I_2.
\end{align}
Given that $\chi_n \to \chi_0$ in $L^1(Q_\nu;\{0,1\})$ and $u_n \to u_0$ in $L^1(Q_\nu;\Rb^N)$, it follows that $\theta_n \to \chi_0(\cdot - x_0)$ in $L^1(Q_\nu(x_0,\e);\{0,1\})$ and $v_n \to u_0(\cdot - x_0)$ in $L^1(Q_\nu(x_0,\e);\Rb^N)$. Thus, 
\begin{equation}\label{Kg2}
I_1 \geq \frac{1}{\e^{N-1}}
\mathcal F(\chi_0(\cdot - x_0),u_0(\cdot - x_0);Q_\nu(x_0,\e)) 
\geq \frac{1}{\e^{N-1}}
m(\chi_0(\cdot - x_0),u_0(\cdot - x_0);Q_\nu(x_0,\e)). 
\end{equation}
On the other hand, the same calculations that were used to prove \eqref{limit0} by means
of hypothesis \eqref{finfty} allow us to conclude that
\begin{equation}\label{Kg3}
\limsup_{\varepsilon \to 0^+}I_2 = 0.
\end{equation}
Hence, from \eqref{Kg1}, \eqref{Kg2} and \eqref{Kg3} we obtain
$$K(\chi^+(x_0),\chi^-(x_0),u^+(x_0),u^-(x_0),\nu(x_0)) \geq 
g(x_0,\chi^+(x_0),\chi^-(x_0),u^+(x_0),u^-(x_0),\nu(x_0)).$$
\end{proof}

\medskip

\noindent\textbf{Acknowledgements}.
The research of ACB was partially supported by National Funding from FCT -
Funda\c c\~ao para a Ci\^encia e a Tecnologia through project UIDB/04561/2020.
JM acknowledges support from CAMGSD and from the Fundação para a Ciência e a Tecnologia through the grant UID/MAT/04459/2020.
EZ is a member of INdAM GNAMPA, whose support is gratefully acknowledged, also through the GNAMPA project 2020, coordinated by Prof. Marco Bonacini. JM and EZ thank the Isaac
Newton Institute for Mathematical Sciences (Cambridge) for its support and hospitality during the programme DNM (Design of New Materials), where this research project started,
supported by EPSRC grant no EP/R014604/1.

\color{black}

\end{document}